\documentclass[a4paper, reqno]{amsart}
\usepackage{amssymb,amsmath, graphicx}
\usepackage{mathdots}
\usepackage{mathrsfs}
\usepackage{wrapfig}
\usepackage{enumerate}
\usepackage{url}
\usepackage{hyperref}
\usepackage[utf8]{inputenc}
\usepackage{tikz}
\usetikzlibrary{matrix, arrows.meta}
\usetikzlibrary{cd}
\usepackage{mathtools}
\usepackage{slashed}

\def\resto#1#2{{
#1\hskip 0.4ex\vline_{\hskip 0.2ex\raisebox{-0,2ex}
{{${\scriptstyle #2}$}}}}}

\def\at#1#2{
#1\hskip 0.25ex\vline_{\hskip 0.25ex\raisebox{-1.5ex}
{{$\scriptstyle#2$}}}}

\def\union{\mathop{\bigcup}}

\def\textmap#1{\mathop{\vbox{\ialign{
                                  ##\crcr
      ${\scriptstyle\hfil\;\;#1\;\;\hfil}$\crcr
      \noalign{\kern 1pt\nointerlineskip}
      \rightarrowfill\crcr}}\;}}
\def\bigtextmap#1{\mathop{\vbox{\ialign{
                                  ##\crcr
      ${\hfil\;\;#1\;\;\hfil}$\crcr
      \noalign{\kern 1pt\nointerlineskip}
      \rightarrowfill\crcr}}\;}}

\def\textlmap#1{\mathop{\vbox{\ialign{
                                  ##\crcr
      ${\scriptstyle\hfil\;\;#1\;\;\hfil}$\crcr
      \noalign{\kern-1pt\nointerlineskip}
      \leftarrowfill\crcr}}\;}}

\def\C{{\mathbb C}}

\def\N{{\mathbb N}}

\def\R{{\mathbb R}}
\def\Z{{\mathbb Z}}

\def\g{{\mathfrak g}}
\def\hg{{\mathfrak h}}

\def\jg{{\mathfrak j}}
\def\kg{{\mathfrak k}}

\def\sg{{\mathfrak s}}
\def\tg{{\mathfrak t}}

\def\Fg{{\mathfrak F}}

\def\Sg{{\mathfrak S}}

\def\Ug{{\mathfrak U}}

\theoremstyle{remark}
\newtheorem{ex}{Example}[section]

\theoremstyle{plain}

\newtheorem{sz}{Satz}[section]
\newtheorem{thry}[sz]{Theorem}
\newtheorem{pr}[sz]{Proposition}

\newtheorem{co}[sz]{Corollary}
\newtheorem{dt}[sz]{Definition}
\newtheorem{lm}[sz]{Lemma}

\theoremstyle{remark}
\newtheorem{re}[sz]{Remark} 

\theoremstyle{plain}

\def\End{\mathrm {End}}

\def\Spin{\mathrm {Spin}}
\def\U{\mathrm{U}}
\def\O{\mathrm{O}}

\def\SO{\mathrm {SO}}

\def\GL{\mathrm {GL}}

\def\gl{\mathrm {gl}}
\def\S{\mathrm {S}}

\def\so{\mathrm {so}}

\def\Iso{\mathrm {Iso}}

\def\Hom{\mathrm{Hom}}

\def\id{ \mathrm{id}}

\def\ad{\mathrm {ad}}
\def\alt{\mathrm {alt}}
\def\U2{\mathrm{U(2)}}
\def\niq{=\kern-.18cm /\kern.08cm}

\def\ad{\mathrm{ad}}

\def\Dr{\slashed{D}}

\newcommand{\cal}{\mathcal}

\makeatletter
\DeclareFontFamily{OMX}{MnSymbolE}{}
\DeclareSymbolFont{MnLargeSymbols}{OMX}{MnSymbolE}{m}{n}
\SetSymbolFont{MnLargeSymbols}{bold}{OMX}{MnSymbolE}{b}{n}
\DeclareFontShape{OMX}{MnSymbolE}{m}{n}{
    <-6>  MnSymbolE5
   <6-7>  MnSymbolE6
   <7-8>  MnSymbolE7
   <8-9>  MnSymbolE8
   <9-10> MnSymbolE9
  <10-12> MnSymbolE10
  <12->   MnSymbolE12
}{}
\DeclareFontShape{OMX}{MnSymbolE}{b}{n}{
    <-6>  MnSymbolE-Bold5
   <6-7>  MnSymbolE-Bold6
   <7-8>  MnSymbolE-Bold7
   <8-9>  MnSymbolE-Bold8
   <9-10> MnSymbolE-Bold9
  <10-12> MnSymbolE-Bold10
  <12->   MnSymbolE-Bold12
}{}

\let\llangle\@undefined
\let\rrangle\@undefined
\DeclareMathDelimiter{\llangle}{\mathopen}%
                     {MnLargeSymbols}{'164}{MnLargeSymbols}{'164}
\DeclareMathDelimiter{\rrangle}{\mathclose}%
                     {MnLargeSymbols}{'171}{MnLargeSymbols}{'171}
\makeatother
\def\add{{\hskip-0.7ex{\scriptscriptstyle \mathrm{ad}}}}
\def\EE{{\hskip-0.7ex{\scriptscriptstyle E}}}
\def\MM{{\hskip-0.7ex{\scriptscriptstyle M}}}

 \def\psp#1#2%
  {\mathop{}%
   \mathopen{\vphantom{#2}}^{#1}%
   \kern-\scriptspace%
    \hskip -0.3mm{#2}} 

\begin{document}

\title{Infinitesimal homogeneity and bundles} 
\author{Arash Bazdar }
\address{Aix Marseille Université, CNRS, Centrale Marseille, I2M, UMR 7373, 13453 Marseille, France}
\email[Arash Bazdar]{bazdar.arash@gmail.com}
\author{Andrei  Teleman}
\email[Andrei  Teleman]{andrei.teleman@univ-amu.fr}

\begin{abstract}

 Let $Q\to M$ be a principal $G$-bundle, and $B_0$ a connection on $Q$.  We introduce  an infinitesimal homogeneity condition for sections in  an associated vector bundle $Q\times_GV$ with respect to $B_0$, and, inspired by the well known Ambrose-Singer theorem, we prove the  existence of a connection   which satisfies a system of parallelism conditions. We explain how this general theorem can be used to prove the known Ambrose-Singer type theorems by an appropriate choice of the initial system of data.
  We also obtain new applications, which cannot be obtained using the known formalisms, e.g. a   classification theorem  for locally homogeneous spinors. 
   
  Finally we introduce natural local homogeneity and local  symmetry conditions for triples $(g,P\textmap{p} M,A)$ consisting of a Riemannian metric on $M$, a principal bundle on $M$, and a connection on $P$.   Our main results concern locally homogeneous and locally symmetric triples, and they can be viewed as bundle versions of the Ambrose-Singer and Cartan theorem.
  \end{abstract}
\maketitle
\begin{quotation}
	
\end{quotation}

\tableofcontents

\section{Introduction}
\label{intro}

A Riemannian manifold $(M,g)$ is called     symmetric (locally symmetric) if, for each point $x\in M$, there exists an isometry $s_x:M\to M$  (respectively a local isometry $s_x:U\to U$ defined on an open neighborhood $U$ of $x$) such that $s_x(x)=x$ and $d_xs_x=-\id_{T_xM}$. 

 An explicit characterization of these spaces is given by the following theorem due to E. Cartan:

\begin{thry}\label{Cartan}\cite[Theorem 6.2, Theorem 6.3]{KN2}
	A Riemannian manifold $(M,g)$ is locally symmetric if and only if $\nabla^gR^g = 0$, where $\nabla^g$ is the Levi-Civita connection of $g$ and $R^g$ is the Riemann curvature tensor. Moreover, any connected, simply connected, complete locally symmetric Riemannian manifold is (globally) symmetric.
\end{thry} 

Symmetric (locally  symmetric) Riemannian manifolds are homogeneous (resp. locally homogeneous).  Recall that a  Riemannian manifold $(M,g)$ is called homogeneous (locally homogeneous) if for any two points $x_1$, $x_2\in M$ there is an isometry $\varphi:M\to M$ (resp.   $\varphi:U_1\to U_2$ between open neighborhoods $U_i\ni x_i$) with $\varphi(x_1)=x_2$.  

Cartan's result (Theorem \ref{Cartan}) has been extended by Ambrose and Singer to the locally homogeneous framework:

\begin {thry} \label{ASTheorem} \cite{AS}\cite{Si}   A Riemannian manifold $(M,g)$ is locally homogeneous if and only if there exists a metric connection $\nabla$ such that $\nabla R^\nabla=\nabla  T^\nabla=0$,
where $T^\nabla\in A^0(L^2_\alt(T_M,T_M))$ and $R^\nabla\in A^2(\gl(T_M))$ denote the torsion and the curvature of  $\nabla$. Moreover, any connected, simply connected, complete locally homogeneous Riemannian manifold is (globally) homogeneous.
\end{thry}

A metric connection  satisfying the above conditions is called an Ambrose-Singer connection  by some authors \cite{TV, Tr, NT}. Obviously the Levi-Civita connection of a locally symmetric space is an Ambrose-Singer connection. Theorem \ref{ASTheorem} can be reformulated as follows: $(M,g)$ is locally homogeneous  if and only if it admits an Ambrose-Singer connection. 

This statement is a special case of a very general principle in modern differential geometry: a local homogeneity condition for a class of geometric structures  on a given manifold is equivalent to the existence of a connection (on a suitable bundle) which satisfies an Ambrose-Singer type condition. The  goal of this article is  a general theorem which implies all known results of this type, i.e. {\it all Ambrose-Singer type theorems}.   

Our approach and our formalism  are motivated by the following bundle versions of Cartan's and  Ambrose-Singer's theorem. In order to state these theorems, which will be proved in section \ref{LHCchap}, we define the bundle analogues of the notions ``(locally) symmetric space", ``(locally) homogeneous space": 

\begin{dt}\label{DefSymLH}
Let $M$ be a differentiable manifold, and $K$ a Lie group. A triple $(g,P\textmap{p} M,A)$ consisting of a Riemannian metric $g$ on $M$, a principal $K$-bundle $P\textmap{p} M$ on $M$, and a connection $A$ on $P$ is called:
\begin{enumerate}
\item Symmetric (locally symmetric)	if, for each point $x\in M$, there exists: 
\begin{enumerate}[(i)]
\item 	An isometry $M\textmap{s_x}M$  (resp. a local isometry $U\textmap{s_x} U$) such that $s_x(x)=x$ and $d_xs_x=-\id_{T_xM}$,
\item An $s_x$-covering   bundle isomorphism $P\textmap{\Phi_x} P$ (resp. $P_U\textmap{\Phi_x} P_U$) which leaves any point $y\in P_x$ and the connection $A$ (resp. $A_U$) invariant.
\end{enumerate}
\item Homogeneous (locally homogeneous)  if for any two points $x_1$, $x_2\in M$ there exists:
 \begin{enumerate}
\item  An isometry $M\textmap{\varphi} M$ (resp. a local isometry  $x_1\in U_1\textmap{\varphi} U_2\ni x_2$) with $\varphi(x_1)=x_2$. 
\item A  $\varphi$-covering bundle isomorphism $P\textmap{\Phi} P$ (resp. $P_{U_1}\textmap{\Phi} P_{U_2}$) which leaves $A$-invariant (resp. such that $\Phi^*(A_{U_2})=A_{U_1}$).
\end{enumerate} 
\end{enumerate}

\end{dt}

In this definition $P_x$ stands for the fiber $p^{-1}(x)$, and for an open set $U\subset M$, we denoted by $P_U$ ($A_U$) the restriction of the bundle $P$ (the connection $A$) to $U$ (respectively $P_U$). 
With these definitions we will prove the following bundle analogous of Cartan's, respectively Ambrose-Singer's theorem :
\begin{thry} \label{LSLHTh}
Let $M$ be a differentiable manifold, and $K$ a compact Lie group. 	 
\begin{enumerate}
\item 	\label{LsymSt} A triple $(g,P\textmap{p} M,A_0)$ as above is locally symmetric if and only if %
$$\nabla^g R^g=0,\ (\nabla^{g}\otimes\nabla^{A_0})F^{A_0}=0.$$
\item \label{LHSt} A triple $(g,P\textmap{p} M,A_0)$ as above is locally homogeneous if and only if there exists a pair $(\nabla,A)$ consisting of a metric connection on $M$ and a connection $A$ on $P$ such that
		$$\nabla  R^\nabla=0,\ \nabla T^\nabla=0,\ (\nabla\otimes \nabla^A)F^A=0,\  (\nabla\otimes \nabla^A)(A-A_0)=0\,.
		$$
\end{enumerate}

\end{thry}

In this statement we used the affine space structure of the space ${\cal A}(P)$ of all connections on $P$, so  $A-A_0$ is an element of $A^1(M,\ad(P))$.  \\

Why are we interested in these bundle analogues of the classical notions and results we have recalled?  Note first that a result of Itoh, which we explain below, yields a large class of symmetric triples in the sense of Definition \ref{DefSymLH}.  
 
\begin{ex}

Let $(G,H)$ be a reductive pair. In other words, $G$ is a connected Lie group, $H$ a closed subgroup of $G$ such that $\hg$ admits an $\ad_H$-invariant complement $\sg$ in $\g$. Such a complement defines a $G$-invariant connection $A^\sg$ on the principal bundle $G\to G/H$. For a  Lie group $K$, the data of a $G$-homogeneous $K$-bundle on $M\coloneqq G/H$ is equivalent to the data of a morphism $\lambda:H\to K$. The bundle corresponding to $\lambda$ is $P_\lambda\coloneqq G\times_\lambda K$. This bundle comes with an obvious  $G$-invariant connection, namely $A^\sg_\lambda\coloneqq (\rho_\lambda)_*(A^\sg)$, where $\rho_\lambda:G\to P_\lambda$ is the obvious bundle morphism.  For a symmetric space $(G,H,\sigma)$ (see \cite[Section XI.2]{KN2})   one has a canonical choice of $\sg$ (see \cite[Proposition 2.2]{KN2}), so any $G$-homogeneous $K$-bundle $P$ on $M$ comes with a canonical connection, which will be denoted by $A_\lambda$. With these preparations we can state the following result of  Itoh \cite{It}.
\begin{thry}\cite{It}\label{ITOH}
Let $(G,H,\sigma)$ be a symmetric space with $H$ compact, and $M=G/H$ compact and oriented. Let $P$ be a $G$-homogeneous principal $K$-bundle on $M$ with $K$ compact. Let $g$ be the Riemannian metric on $M$	associated with an $\ad_H$-invariant inner product on the canonical complement $\sg$ of $\hg$ in $\g$. Then, the curvature of the canonical invariant connection $A_\lambda$ on $P$ is parallel with respect to $\nabla^g\otimes \nabla^{A_\lambda}$. \end{thry}
 Therefore, the triple $(g,P_\lambda\to M, A_\lambda)$ is locally symmetric; it will be symmetric when $M$ is simply connected. 
\end{ex}

Note also that 

\begin{ex} For a locally symmetric (locally homogeneous) Riemannian manifold $(M,g)$, let $C_0\in {\cal A}(\O(M))$ denote the Levi-Civita connection on the orthonormal frame bundle $\O(M)$ of $(M,g)$. The associated triple $(g, \O(M)\to M, C_0)$ is locally symmetric (respectively locally homogeneous) in the sense of Definition  \ref{DefSymLH}. 	
\end{ex}

 Our bundle versions of Cartan's and Ambrose-Singer's theorem (Theorem \ref{LSLHTh}) have important consequences : we will prove (see Theorem \ref {CompleteBase}, Corollary \ref{CompleteBasecor}) that any locally symmetric (homogeneous) triple on a simply connected, complete base $(M,g)$  is globally symmetric (homogeneous). This will allow us to prove a classification theorem for locally symmetric (homogeneous) triples:   any locally symmetric (homogeneous) triple on a compact base $M$ is a quotient of a globally symmetric (homogeneous) triple on the universal cover $\tilde M$ (see Theorem \ref {main-DIffcase}, Corollary \ref{main-coro}). These results play an important role and can be used effectively for the classification of  geometric  manifolds (in the sense of Thurston) which are  principal bundles over a given geometric base (see \cite{Ba2}, \cite{BaTe}). \\

In this article we prove a  general theorem, explained in the next subsection, which yields not only Theorem \ref{LSLHTh} (\ref{LHSt}), but {all}  Ambrose-Singer type theorems which we found in the literature. Moreover, as special cases of our theorem, we will obtain new such results, for instance a characterization of  locally homogeneous pairs $(g,s)$ consisting of a Riemannian metric on a spin manifold and a spinor. 

\subsection{Infinitesimally homogeneous sections}

Let $M$ be a differentiable  manifold of dimension $n$, and $L(M)\to M$ be its bundle of linear frames. Let $G$ be a Lie group, $\pi:Q\to M$ be a principal $G$-bundle over $M$, $r:G\to \GL(n)$ be a morphism of Lie groups, and $ f : Q\to L(M)$ be an $r$-morphism of principal bundles over $M$. Let also $V$ be a finite dimensional vector space,    $\rho:G\to \GL(V)$ be a morphism of Lie groups, and $E\coloneqq Q\times_\rho V$ be the associated vector bundle. \vspace{1.5mm}

For $B\in \cal A (Q)$, the linear connection on the tangent bundle $T_M$ associated with $f_*(B)$ is denoted by $\nabla^B_\MM$; it coincides with the linear connection associated with $B$ on the vector bundle $Q\times_r \R^n\simeq T_M$.  The linear connections on the vector bundles $E=Q\times_\rho V$  ($\ad(Q)=Q\times_\ad \g$) associated with $B$ will be denoted by $\nabla^B_\EE$ (respectively $\nabla^B_\add$).
\vspace{2mm} \\ 
Let $B_0$ be a fixed connection on $Q$, and $\sigma\in\Gamma(E)$.  We define recursively  the derivatives of $\sigma$ with respect to $B_0$ by
$$
\sigma^{(0)}_{B_0}=\sigma,\ \sigma^{(i)}_{B_0}=\big((\nabla^{B_0}_\MM)^{\otimes (i-1)}\otimes \nabla^{B_0}_\EE\big)\sigma^{(i-1)}_{B_0}\,.
$$
 For any $k\in\N$ and $x\in M$  put 
\begin{equation}\label{hsigmax}\hg^\sigma_x{(k)}\coloneqq \{b\in \ad(Q_x)|\ b\cdot\{\sigma^{(i)}_{B_0}\}_x=0 \hbox{ for } 0\leq i\leq k\}\,,
\end{equation}
and note that $\hg^\sigma_x{(k)}$ is a Lie subalgebra of the finite dimensional Lie algebra $\ad(Q_x)$. One has $\hg^\sigma_x{(k+1)}\subset \hg^\sigma_x{(k)}$ for any $k$. Put 
$$k_x^\sigma\coloneqq \min\{k\in\N|\ \hg^\sigma_x{(k+1)}=\hg^\sigma_x{(k)}\}\,.$$

If $(x_1,x_2)\in M\times M$ then for any $G$-equivariant isomorphism $\theta:Q_{x_1}\to Q_{x_2}$, we can define the next linear isomorphisms
$$\theta_V:E_{x_1}\to E_{x_2}\,,\ \theta_{\g}:\ad(Q_{x_1})\to\ad(Q_{x_2})\,. $$
We also obtain a linear isomorphism  $\theta_M:T_{x_1}M\to T_{x_2}M$ via the bundle morphism $ f $.
Now, denoting by $\theta^k$ the induced isomorphism $\{\Lambda^{1}_{x_1}\}^{\otimes k}\otimes E_{x_1}\to \{\Lambda^{1}_{x_2}\}^{\otimes k}\otimes E_{x_2}$ we define
 
\begin{dt} \label{InfHomSectDef} Let $B_0\in {\cal A}(Q)$ be a connection on $Q$. A section $\sigma\in\Gamma(E)$  is called infinitesimally 	homogeneous  with respect to $B_0$ if for any pair $(x_1,x_2)\in M\times M$ there exists  a $G$-equivariant isomorphism $\theta:Q_{x_1}\to Q_{x_2}$ such that 
\begin{equation}\label{LocHomCond1}
\theta^i\big(\big\{\sigma^{(i)}_{B_0} \big\}_{x_1}\big)=\big\{\sigma^{(i)}_{B_0}\big\}_{x_2} \hbox{  for }  0\leq  i\leq k_{x_1}^\sigma+1\,.	
\end{equation}
\end{dt}

Let $\theta:Q_{x_1}\to Q_{x_2}$ be a $G$-equivariant isomorphism such that (\ref{LocHomCond1}) holds. Then  $\theta_\g$ applies isomorphically $\hg^\sigma_{x_1}{(k)}$ on $\hg^\sigma_{x_2}{(k)}$ for $0\leq k\leq k_{x_1}^\sigma+1$. This implies
\begin{re}
 If  $\sigma\in\Gamma(E)$  is an  infinitesimally homogeneous section  with respect to $B_0$, then $k_x^\sigma$ is independent of $x$. The obtained constant will be denoted by $k^\sigma$. \end{re}
 %
 

Our  generalizations of the Ambrose-Singer theorem  will be based on:
 
 \newtheorem*{th-main}{Theorem \ref{ConnectionB-Th}}
  \begin{th-main}\label{ConnectionB-intro}
Suppose that $\sigma$ is  infinitesimally 	homogeneous  with respect to  $B_0$. Fix $q_0\in Q$, and suppose that the pair $(G,H^\sigma_{q_0})$ is reductive. There exists a connection $B\in {\cal A}(Q)$ such that:
\begin{enumerate}
\item $((\nabla^{B}_\MM)^{\otimes k}\otimes \nabla^{B}_\EE) (\sigma^{(k)}_{B_0})=0\ \hbox{ for } \ 0\leq k\leq k^\sigma+1$.
\item 	$(\nabla^{B}_{\MM}\otimes\nabla^B_\add)(B-B_0)=0$.
\end{enumerate}
	
\end{th-main}
Note that (supposing $M$ connected) the existence of a connection $B$ satisfying (1) implies the infinitesimally 	homogeneity of $\sigma$  with respect to  $B_0$. Indeed, for any pair $(x_1,x_2)\in M\times M$, parallel transport with respect to $B$ along a path  $\gamma:[0,1]\to M$  with $\gamma(0)=x_1$, $\gamma(1)=x_2$ defines a $G$-equivariant diffeomorphism $Q_{x_1}\to Q_{x_2}$ mapping $\{\sigma^{(k)}_{B_0} \big\}_{x_1}$ onto $\{\sigma^{(k)}_{B_0} \big\}_{x_2}$ for  any $0\leq  k\leq k_{x_1}^\sigma+1$.\\

The reductivity condition needed in the theorem (see section \ref{EXistenceAdConn}) is automatically satisfied when $G$ is compact.
We will see that this theorem can be used to prove Theorem \ref{LSLHTh} stated above. Moreover, it provides a general approach  for the study of locally homogeneous  objects by investigating their corresponding infinitesimally homogeneous analogues. In each specific case we will choose an appropriate system of data 
$$(M,Q,\rho, f:Q\to L(M), B_0,\sigma)$$
 and, using Theorem \ref{ConnectionB-Th}, we will prove in each case a result of the type ``locally homogeneous is equivalent to infinitesimally homogeneous, and is equivalent to globally homogeneous  in the complete, simply connected case". \vspace{2mm}

Section \ref{InfHomogSect} of this article is devoted to the proof of Theorem \ref{ConnectionB-intro}.   In section \ref{ApplicationS}, this result will be first used to prove known Ambrose-Singer type theorems:
\begin{enumerate}
\item The classical Ambrose-Singer theorem (Theorem \ref{ASTheorem}).
\item  Kiri\v{c}enko's theorem for locally homogeneous systems $(g,P_1,\dots,P_k)$, where $g$ is a Riemannian metric, and $P_i$ are tensor fields (Theorem \ref{KiriTh}). In particular we will consider the case of locally homogeneous almost Hermitian manifolds (Theorem \ref{AS-Local-verionHer}).
\item  Opozda's theorem on locally homogeneous $G$-structures (Theorem \ref{Opzda}).
\end{enumerate}

Note that, using a lemma based on elliptic regularity (see Lemma \ref{CoAnnConn2}), we show that the analyticity condition in  Opozda's theorem is not necessary. 

At the end of the section we prove a new Ambrose-Singer type theorem which concerns  locally homogeneous pairs $(g,s)$ consisting of a Riemannian metric and a spinor on a spin Riemannian manifold $M$. This result is not a special case of Kiri\v{c}enko's theorem, because a spinor on $M$ cannot be considered as a tensor field. Indeed, $s$ is a section in the spinor bundle, which is a vector bundle associated with the principal bundle of the fixed spin structure, not with the frame bundle $L(M)$.  This shows already the advantage of our formalism (which uses a general principal $G$-bundle $Q$  instead of the classical $L(M)$).

 Section 4 is dedicated to locally symmetric (homogeneous) triples in the sense of Definition \ref{DefSymLH}. We will prove Theorem \ref{LSLHTh}, and we will explain how this result can be used for the classification of the locally symmetric (homogeneous) triples on a given compact Riemannian manifold in terms of globally symmetric (homogeneous) triples on its universal cover.  The last subsection is dedicated to local homogeneity in the framework of $\Spin^c$-manifolds. We introduce and study natural  local homogeneity conditions for   pairs $(\Lambda,A)$ consisting of a $\Spin^c$-structure and a Levi-Civita lifting $\Spin^c$-connection, and for  triples $(\Lambda,A,s)$ where $(\Lambda,A)$ is a above, and $s$ is a $\Spin^c$-spinor.

\section{An Ambrose-Singer theorem for infinitesimally homogeneous sections}\label{InfHomogSect}

In this section we will give the proof    Theorem \ref{ConnectionB-Th}.

\subsection{A Leibniz formula}

Let $M$ be a differentiable  manifold of dimension $n$ and $L(M)\to M$ be its frame bundle. Let $G$ be a Lie group and let $r:G\to \GL(n)$ be a morphism of Lie groups,  $\pi:Q\to M$ be a principal $G$-bundle over $M$.
%
%
 Moreover, let $V$ be a finite dimensional vector space,    $\rho:G\to \GL(V)$ be a morphism of Lie groups, and $E\coloneqq Q\times_\rho V$ be the associated vector bundle. Put
$$W_{ijpq}\coloneqq (\R^n) ^{\otimes i}\otimes (\R^{n*})^{\otimes j} \otimes  V^{\otimes p}\otimes  V^{*\otimes q}\,,
$$
and let $R:G\to \GL(W_{ijpq})$ be the linear representation induced by $r$ and $\rho$. Let $\g$ be the Lie algebra of $G$ and consider the Lie algebra morphism $\g\to \gl(W_{ijpq})$. This infinitesimal action defines a $G$-invariant pairing
\begin{equation}\label{pairingTensors}
\g\times 	W_{ijpq}\to W_{ijpq}\,,
\end{equation}
which induces a paring of associated vector bundles
\begin{equation}\label{pairingBundles}\ad(Q)\times \big(T_M^{\otimes i}\otimes (\Lambda^1_M)^{\otimes j}\otimes E^{\otimes p}\otimes E^{\otimes *q} \big) \to   T_M^{\otimes i}\otimes (\Lambda^1_M)^{\otimes j}\otimes E^{\otimes p}\otimes E^{\otimes *q}  
\end{equation}
The pairings (\ref{pairingTensors}), (\ref{pairingBundles}) will be denoted by $(b,\eta)\mapsto b\cdot\eta$ to save on notations. For instance, if $b\in \ad(Q)$ and $\eta\in \{\Lambda^{1}_M\}^{\otimes j}\otimes  E $ then the pairing
\begin{equation}\label{pairingSpecial} \ad(Q)\times_M \big( \{\Lambda^{1}_M\}^{\otimes j}\otimes  E\big)\textmap{\cdot}  \{\Lambda^{1}_M\}^{\otimes j}\otimes  E	
\end{equation}
is given by the formula
$$(b\cdot \eta)(w_1,\dots,w_j)\coloneqq b\cdot  (\eta(w_1,\dots,w_j))-\sum_{i=1}^j \eta(w_1,\dots, b\cdot w_i,\dots,w_j)\,.
$$
The fact that the pairing (\ref{pairingTensors}) is $G$-invariant, has an important consequence:
\begin{re}\label{parallelPairing}
The pairing of associated vector bundles given in (\ref{pairingBundles}) is parallel with respect to any connection   on $Q$.	
\end{re}
We shall also need the pairing
\begin{equation}\label{new-pairing}
\big(\{\Lambda^1_M\}^{\otimes k}\otimes\ad({Q})\big)\times_M \big( \{\Lambda^{1}_M\}^{\otimes j}\otimes  E\big)\textmap{\cdot}  \{\Lambda^{1}_M\}^{\otimes (j+k)}\otimes  E	
\end{equation}
given by
$$\big(u\otimes b)\cdot \eta\coloneqq u\otimes   (b\cdot \eta)\,.
$$
Explicitly, for  $\beta\in \{\Lambda^1_x\}^{\otimes k}\otimes\ad({Q}_x)$ and $\eta\in \{\Lambda^{1}_x\}^{\otimes j}\otimes  E_x$, one has
\begin{equation}\label{new-pairing-formula}
\begin{split}
&(\beta\cdot \eta) (v_1,\dots,v_k,w_1,\dots,w_j)=\big\{\beta(v_1,\dots,v_k)\cdot \eta\big\}(w_1,\dots,w_j) =\\
&= \beta(v_1,\dots,v_k)\cdot  \eta(w_1,\dots,w_j) -\sum_{i=1}^j \eta (w_1,\dots, \beta(v_1,\dots,v_k)\cdot w_i,\dots,w_j) 
\,. 
\end{split}	
\end{equation}
Note also that, if  $\beta=(\omega_1\otimes\dots\otimes\omega_k)\otimes b $ is a tensor monomial  with $\omega_i\in \Lambda^1_x$ and $b\in\ad(Q_x)$, then
$$\big((\omega_1\otimes\dots\otimes\omega_k)\otimes b\big)\cdot \eta=(\omega_1\otimes\dots\otimes\omega_k)\otimes (b\cdot \eta)\,.
$$

Let $B$ be a connection on the principal $G$-bundle $Q$ as above and  $ f : Q\to L(M)$ be an $r$-morphism of principal bundles over $M$. Let $\nabla^B_\MM$ denote the linear connection on the tangent bundle $T_M$ associated with $f_*(B)$. This connection corresponds  to the linear connection associated with $B$ on the vector bundle $Q\times_r \R^n\simeq T_M$. Let $\nabla^B_\EE$, $\nabla^B_\add$ denote the linear connections on the vector bundles $E=Q\times_\rho V$, respectively $\ad(Q)=Q\times_\ad \g$ which are associated with $B$. We also need the linear connection $\nabla^B_{ijpq}$ on the vector bundle $T_M^{\otimes i}\otimes (\Lambda^1_M)^{\otimes j}\otimes E^{\otimes p}\otimes E^{\otimes *q}$ associated with $B$. This connection can be written as
$$\nabla^B_{ijpq}=(\nabla^B_\MM)^{\otimes i}\otimes (\nabla^B_\MM)^{*\otimes j}\otimes (\nabla^B_\EE)^{\otimes p}\otimes (\nabla^B_\EE)^{*\otimes q}\,.
$$
Taking into account Remark \ref{parallelPairing} it follows
\begin{re}\label{LeibnizRe}	
For any $x\in M$ and tangent vector $\xi\in T_xM$ the following Leibniz rule  holds:
$$\nabla^B_{ijpq,\xi}(b\cdot\eta)=(\nabla^B_{\add,\xi} b)\cdot\eta +b\cdot \nabla^B_{ijpq,\xi}\eta\,.
$$
\\
In particular, for any pairs $(b,v)\in \Gamma(\ad({Q}))\times {\mathfrak{X} }(M)$ one has
$$\nabla^B_{\MM,\xi}(b\cdot v)=	(\nabla^B_{\add,\xi} b)\cdot v+b\cdot (\nabla^B_{\MM,\xi} v)\,,
$$
and for any $(b,y)\in \Gamma(\ad(Q))\times\Gamma(E)$
$$\nabla^B_{\EE,\xi}(b\cdot y)=	(\nabla^B_\add b)\cdot y+b\cdot (\nabla^B_\EE y)\,.
$$
\end{re}

The tensor product of connections induces the connection  $(\nabla^{B}_\MM)^{\otimes j}\otimes \nabla^{B}_\EE$ on the vector bundle $\{\Lambda^{1}_M\}^{\otimes j}\otimes  E $ and the connection $(\nabla^{B}_\MM)^{\otimes j}\otimes \nabla^{B}_\add$ on the vector bundle  $\{\Lambda^{1}_M\}^{\otimes j}\otimes  \ad({Q}) $.
%
The pairing (\ref{new-pairing}) will be used to obtain the following variation formula for the connection $(\nabla^{B}_\MM)^{\otimes j}\otimes \nabla^{B}_\EE$  with respect to $B$: 
\begin{re}\label{variation}
Let $B\in {\cal A}(Q)$ and $\beta\in A^1(M,\ad(Q))$. Put $B'\coloneqq B+\beta$. For any $\eta\in \Gamma \big( \{\Lambda^{1}_M\}^{\otimes j}\otimes  E\big)$ one has
\begin{equation}\label{variationEq}
\big((\nabla^{B'}_\MM)^{\otimes j}\otimes \nabla^{B'}_\EE\big)\eta	=\big((\nabla^{B}_\MM)^{\otimes j}\otimes \nabla^{B}_\EE\big)\eta + \beta\cdot \eta\,.
\end{equation}

\end{re}
 Explicitly, for any $x\in M$ and $\xi\in T_xM$ 
$$\big((\nabla^{B'}_\MM)^{\otimes j}\otimes \nabla^{B'}_\EE\big)_\xi\eta=\big((\nabla^{B}_\MM)^{\otimes j}\otimes \nabla^{B}_\EE\big)_\xi\eta+\beta(\xi)\cdot \eta\,.
$$

With the notations introduced above one can prove the following Leibniz formula.
\begin{lm} \label{LeibnizLemma} For any 	$\beta\in \Gamma\big(\{\Lambda^1_M\}^{\otimes k}\otimes\ad({Q})\big)$, $\eta\in \Gamma\big( \{\Lambda^{1}_M\}^{\otimes j}\otimes  E \big)$, and any tangent vector $\xi\in T_xM$ one has
\begin{equation}\label{LeibnizFormula}
\big((\nabla^{B}_\MM)^{\otimes (k+j)}\otimes \nabla^{B}_\EE\big)_\xi\big(\beta\cdot \eta\big)=\big(\big((\nabla^{B}_\MM)^{\otimes k}\otimes \nabla^{B}_\add\big)_\xi \beta\big)\cdot \eta+ \beta\cdot \big((\nabla^{B}_\MM)^{\otimes j}\otimes \nabla^{B}_\EE\big)_\xi\eta\,.	
\end{equation}
\begin{proof} We give the proof in the case $k=1$, which will be used later. We may suppose that $\beta$ is tensor monomial, so it has the form $\beta=\omega\otimes b$, where $\omega\in\Gamma(\Lambda^1_M)$, and $b\in \Gamma(\ad(Q))$. Using Remark \ref{LeibnizRe} we obtain
\begin{equation*}
\begin{split}
\big((\nabla^{B}_\MM)^{\otimes (1+j)}\otimes  \nabla^{B}_\EE  \big)_\xi &\big(\beta\cdot \eta\big)=\big((\nabla^{B}_\MM)^{\otimes (1+j)}\otimes \nabla^{B}_\EE\big)_\xi\big((\omega\otimes b)\cdot \eta\big)\\
=& \big((\nabla^{B}_\MM)\otimes((\nabla^{B}_\MM)^{\otimes j}\otimes \nabla^{B}_\EE)\big)_\xi\big(\omega\otimes (b\cdot \eta)\big)\\
=&\nabla^{B}_{\MM,_\xi}\omega\otimes(b\cdot\eta)+\omega\otimes\big((\nabla^{B}_{\add,\xi}b)\cdot \eta+b\cdot((\nabla^{B}_\MM)^{\otimes j}\otimes \nabla^{B}_\EE)_\xi\eta\big)\\
=&\big(\nabla^{B}_{\MM,_\xi}\omega\otimes b+\omega\otimes \nabla^{B}_{\add,\xi}b\big)\cdot \eta +(\omega\otimes b)\cdot ((\nabla^{B}_\MM)^{\otimes j}\otimes \nabla^{B}_\EE)_\xi\eta\big)\\
=&\big(\big(\nabla^{B}_\MM \otimes \nabla^{B}_\add\big)_\xi \beta\big)\cdot \eta+\beta\cdot ((\nabla^{B}_\MM)^{\otimes j}\otimes \nabla^{B}_\EE)_\xi\eta\big)\,.
\end{split}
\end{equation*}

\end{proof}
	
\end{lm}

\subsection{The existence of an adapted connection}\label{EXistenceAdConn}

\def\Gr{\mathbb{G}}
\def\codim{\mathrm{codim}}

Let $Q$ be a principal $G$-bundle on $M$, and $H$ be a closed subgroup of $G$. We will need the following   general result which classifies all the reductions of $Q$ to subgroups of $G$ which are conjugate to $H$. For a class $u=gH\in G/H$, denote by $H_u$ the conjugate $H_u\coloneqq gHg^{-1}=\iota_g(H)$ (which obviously depends only on $u=[g]_H$). In general, for a section $\varphi$ in an associated bundle with fiber $F$ we will denote by $\tilde \varphi$ the associated $F$-valued equivariant map.

\begin{lm} \label{RD-Lemma}
	Let $\pi:Q\to M$ be a principal $G$-bundle over a  manifold $M$, and let $H$ be a closed   subgroup of $G$. 
	\begin{enumerate}
	\item 	There is a bijection between the set of pairs $(u,P)$, where $P$ is a $H_u$-reduction of $Q$, and the   product $ (G/H)\times \Gamma(M,Q\times_{G} (G/H))$. The $H_u$-reduction associated with a pair $(u,\varphi)$ is the pre-image $\tilde\varphi^{-1}(u)$. 
	\item Let  	$\varphi\in \Gamma(M,Q\times_{G} (G/H))$, and $y\in Q$. The pre-image $\tilde\varphi^{-1}(\tilde\varphi(y))$ is an $H_{\tilde \varphi(y)}$-reduction of $Q$ which contains $y$. The subgroup $H_{\tilde \varphi(y)}$ coincides with the stabilizer of $\tilde \varphi(y)$ with respect to the left $G$-action on $G/H$.
	\end{enumerate}
\end{lm}
\begin{proof}

Let $K\subset G$  be a closed subgroup. By \cite[Proposition 5.6]{KN} we have a bijection between $\Gamma(M,Q\times_{G} (G/K))\times (G/K)$ and the set of $K$-reductions of $Q$ given by
$$
\varphi\mapsto \tilde\varphi^{-1}([e]_K)\,.
$$
Letting $u$ vary in $G/H$ we obtain a bijection $a:{\cal S}\to {\cal R}$ between the set of pairs
$$
{\cal S}\coloneqq  \big\{(u,\psi)| \ u\in G/H,\ \psi\in \Gamma(M,Q\times_{G} (G/H_u))\}  
$$
and the set of pairs
$$
{\cal R}\coloneqq \big\{(u,P)|\ u\in G/H,\ P \hbox{ is an $H_u$-reduction of $Q$}\big\}
$$
given by  $a(u,\psi)=\tilde\psi^{-1}([e]_{H_u})$. Note now that an element $u=[g]_H\in G/H$ defines a diffeomorphism of left $G$-spaces 
$$
r_u:G/H_u\to G/H\,,\ r_u([\gamma]_{H_u})\coloneqq [\gamma g]_H.$$
We obtain   a bijection
$$
b:(G/H)\times \Gamma(M,Q\times_{G} (G/H))\to {\cal S}\,,\ b(u,\varphi)\coloneqq (u,\psi_{u,\varphi})\,,
$$
where $\psi_{u,\varphi}\in \Gamma(M,Q\times_{G} (G/H_u))$ is the section defined by the $H_u$-equivariant map
$$
\tilde\psi_{u,\tilde\varphi}\coloneqq r_u^{-1}\circ \tilde\varphi\,.
$$
It is easy to check that $\tilde\psi_{u,\tilde\varphi}$ is $H_u$-equivariant if and only if $\tilde \varphi$ is $H$-equivariant. The composition $a\circ b$ will be a bijection $(G/H)\times \Gamma(M,Q\times_{G} (G/H))\to {\cal R}$ as claimed.
\end{proof}

Let $P\subset Q$ be the $H$-reduction of $Q$. A connection $B\in {\cal A}(Q)$ will be called compatible with $P$ if the restriction of $B$ to $P$ is tangent to $P$ (and hence defines a connection on $P)$.
\begin{re}\label{RemRed}
Let  $P\subset Q$ be the $H$-reduction associated with the pair  $(\varphi,u)$.	Then 
the direct image map \cite{KN} defines an affine embedding ${\cal A}(P)\hookrightarrow	{\cal A}(Q)$ whose image is the space of connections on $Q$ which are compatible with $P$, and whose associated linear map
is the inclusion $A^1(\ad(P))\subset A^1(\ad(Q))$.

\end{re}

Taking into account Remark \ref{RemRed}, in the presence of a fixed $H$-reduction $P\subset Q$, we will identify
a connection $B\in {\cal A}(P)$ with its image in ${\cal A}(Q)$, so any connection $B\in {\cal A}(P)$ will also be regarded as connection on $Q$.
 \vspace{3mm}

Let $\sigma$ be an infinitesimally homogeneous section with respect to $B_0\in {\cal A}(Q)$. We will define a closed subgroup $H^\sigma\subset G$ and a $H^\sigma$-reduction $P^\sigma \subset Q$ such that for any connection $B'$ on $P^\sigma$ one has 
$$((\nabla^{B'}_\MM)^{\otimes k}\otimes \nabla^{B'}_\EE)  \sigma^{(k)}_{B_0} =0\ \hbox{ for } 0\leq k\leq k^\sigma+1\,. $$
For $x\in M$ we will identify the fibre $L(M)_x$ of the frame bundle  $L(M)$ with the space of linear isomorphisms $\R^n\to T_xM$. Therefore, with the notations introduced at the beginning of this section, a point $q\in Q$ defines a linear isomorphism 
$$f(q):\R^n\to T_xM\,.$$
Using the $k$-order covariant derivative   $\sigma^{(k)}_{B_0}$ of $\sigma$ we obtain a $G$-equivariant map 
$$\varphi_k:Q\to L^k(\R^n, V)$$
defined by   the formula
$$\sigma^{(k)}_{B_0}(f(q)(\xi_1),\dots,f(q)(\xi_k))=[q,\varphi_k(q)(\xi_1,\dots,\xi_k)]\ \forall(\xi_1,\dots,\xi_k)\in(\R^n)^k\,.$$
In other words, $\varphi_k$ is the $G$-equivariant map $Q\to L^k(\R^n, V)$ associated with $\sigma^{(k)}_{B_0}$ regarded as a section in the associated bundle 
$$(\Lambda^1_M)^{\otimes k}\otimes E=Q\times_{G}  L^k(\R^n, V)\,.$$
Put $W\coloneqq \bigoplus_{k=0}^{k^\sigma+1}L^k(\R^n,V)$ and define a $G$-equivariant map $\Phi:Q\to W$ by 
	\begin{equation}\label{PHI-mq}
		\Phi(q)\coloneqq   (\varphi_k(q))_{0\leq k\leq k^\sigma+1}\,.
	\end{equation}
	 Since the section $\sigma\in\Gamma(E)$ is infinitesimally homogeneous, it follows that $\Phi(Q)$ is a  single $G$-orbit of $W$. Indeed, let $q_0\in Q$, and put $x_0\coloneqq \pi(q_0)$. For a  point $q\in Q$, let $x=\pi(q)$ and $\theta:Q_{x_0}\to Q_{x}$ be a $G$-equivariant isomorphism such that 
	\begin{equation} 
\theta^k\big(\big\{\sigma^{(k)}_{B_0} \big\}_{x_0}\big)=\big\{\sigma^{(k)}_{B_0}\big\}_{x} \hbox{  for }  0\leq  k\leq k^\sigma+1\,.	
\end{equation}
(see Definition \ref{InfHomSectDef}). This implies the equality
$$[\theta(q_0),\varphi_k(q_0)]=[q,\varphi_k(q)]
$$
in $Q_x\times_GL^k(\R^n, V)$. Choosing $a\in G$ such that $\theta(q_0)=qa$, we obtain  
$$\varphi_k(q)=a\varphi_k(q_0) \hbox{ for } 0\leq k\leq k^\sigma+1\,,$$
which shows that  $\Phi(q)\in G \Phi(q_0)$. Therefore $\Phi(Q)\subset G \Phi(q_0)$. Using the $G$-equivariance property of $\Phi$ we get $\Phi(Q)= G \Phi(q_0)$, as claimed.  Put
$$H^\sigma_{q_0}\coloneqq G_{\Phi(q_0)}\,,\ P^\sigma_{q_0}\coloneqq \Phi^{-1}(\Phi(q_0))\,.
$$
Using Lemma \ref{RD-Lemma}, it follows that $P^\sigma_{q_0}$ is a $H^\sigma_{q_0}$-reduction of $Q$.
Since $\Phi$ is obviously constant on $P^\sigma_{q_0}$, it follows that the restrictions $\resto{\varphi_k}{P^\sigma_{q_0}}$ are all constant on $P^\sigma_{q_0}$, so the corresponding sections will be parallel with respect to any connection  on $P^\sigma_{q_0}$. Therefore  
\begin{re}\label{parallel-comp-conn}
The sections $\sigma^{(k)}_{B_0}$ ($0\leq k\leq k^\sigma+1$) are  parallel with respect to any connection  on the bundle $P^\sigma_{q_0}$, so with respect to any connection on $Q$ which  is compatible  with $P^\sigma_{q_0}$.	In other words, for any $B\in {\cal A}(P^\sigma_{q_0})$ we have
\begin{equation}\label{sigma-k-parallel}
((\nabla^{B}_\MM)^{\otimes k}\otimes \nabla^{B}_\EE) (\sigma^{(k)}_{B_0})=0\ \hbox{ for } \ 0\leq k\leq k^\sigma+1\ 
\end{equation}

\end{re}

\begin{re}\label{NewRem}
For any $x\in M$ one has $\hg^\sigma_x(k^\sigma+1)=\ad(P^\sigma_{q_0})_x$.	
\end{re}
\begin{proof}
By definition $\hg^\sigma_x(k^\sigma+1)$ is the infinitesimal stabilizer in $\ad(Q_x)$ of  the system $(\{\sigma^{(k)}_{B_0}\}_x)_{0\leq k\leq k^\sigma+1}$, regarded as an element of the fiber $Q_x\times_G W$ of the associated vector bundle $Q\times_G W$. An element $q\in Q_x$ identifies $\ad(Q_x)$ with $\g$, the system $(\{\sigma^{(k)}_{B_0}\}_x)_{0\leq k\leq k^\sigma+1}$ with $\Phi(q)$  and the fiber $Q_x\times_GW$ with $W$. For $q\in P^\sigma_{q_0}\coloneqq \Phi^{-1}(\Phi(q_0))$, we have $\Phi(q)=\Phi(q_0)$, so  
$$\hg^\sigma_x(k^\sigma+1)=\{[q,u]|\ u\in \g_{\Phi(q_0)}\}=\{[q,u]|\ u\in \hg^\sigma_{q_0}\}=\ad(P^\sigma_{q_0})_x.
$$  
\end{proof}
Remark \ref{NewRem}  shows in particular that the union 
$$\hg^\sigma\coloneqq \union_{x\in M}  \hg^\sigma_x(k^\sigma+1)$$
 is a Lie algebra sub-bundle of $\ad(Q)$.

\begin{pr} \label{new-prop}
Suppose that $\sigma$ is  infinitesimally 	homogeneous  with respect to  $B_0$. Let $B'\in {\cal A}(Q)$ be a  connection  compatible with $P^\sigma_{q_0}$.
Then
\begin{enumerate}
\item The sub-bundle $\hg^\sigma\subset\ad(Q) $
 is  $\nabla^{B'}_{\ad}$-parallel.
\item  \label{newprop2}   One has
$ (\nabla^{B'}_{\MM}\otimes \nabla^{B'}_{\add})(B'-B_0)\in \Gamma(\Lambda^1_M\otimes\Lambda^1_M \otimes \hg^\sigma)\,.
$
\end{enumerate}

\end{pr}
\begin{proof} (1) Let $\nu:[0,1]\to M$ be a smooth path in $M$. By  Remarks \ref{LeibnizRe} and \ref{parallel-comp-conn}, the parallel transport with respect to the connection 	$\nabla^{B'}_{\ad}$  maps isomorphically $\hg^\sigma_{\nu(0)}$ onto  $\hg^\sigma_{\nu(1)}$.
\vspace{2mm}\\
(2) Put $\beta\coloneqq B'-B_0\in A^1(\ad(Q))$. Using Remark \ref{variation} we obtain, for  $0\leq k\leq k^\sigma+1$
\begin{equation}\label{Var} 
0=((\nabla^{B'}_\MM)^{\otimes k}\otimes \nabla^{B'}_\EE)  \sigma^{(k)}_{B_0} =((\nabla^{B_0}_\MM)^{\otimes k}\otimes \nabla^{B_0}_\EE)  \sigma^{(k)}_{B_0} +\beta\cdot \sigma^{(k)}_{B_0}=\sigma^{(k+1)}_{B_0}+\beta\cdot \sigma^{(k)}_{B_0}\,.
\end{equation}
Let $x\in M$, $\xi\in T_x M$. Taking  $0\leq k\leq k^\sigma$,   applying $((\nabla^{B'}_\MM)^{\otimes k}\otimes \nabla^{B'}_{\EE})_\xi$ on both terms of (\ref{Var}), noting that for these values of $k$  the first term on the right will still vanish, and using the Leibniz rule (Lemma \ref{LeibnizLemma}), one obtains
$$\big((\nabla^{B'}_{\MM}\otimes \nabla^{B'}_{\add})_\xi \beta\big)\cdot \sigma^{(k)}_{B_0}=0 \ \hbox{ for }  0\leq k\leq k^\sigma\,.
$$
 Taking into account formula (\ref{new-pairing-formula}) it follows that, for any $v\in T_xM$ one has 
 $$\big(\big((\nabla^{B'}_{\MM}\otimes \nabla^{B'}_{\add})_\xi \beta\big)(v)\big)\cdot \sigma^{(k)}_{B_0}=0\,.$$
  Therefore  for any $(\xi,v)\in T_xM\times T_xM$ one has  
  $$\big((\nabla^{B'}_{\MM}\otimes \nabla^{B'}_{\add}) \beta\big)(\xi,v)\in \hg^\sigma_x(k^\sigma_x)=\hg^\sigma_x(k^\sigma_x+1)\,,$$
which shows that  
$$(\nabla^{B'}_{\MM}\otimes \nabla^{B'}_{\add}) \beta\in   \Gamma(\Lambda^1_M\otimes\Lambda^1_M \otimes \hg^\sigma(k^\sigma+1))=\Gamma(\Lambda^1_M\otimes\Lambda^1_M \otimes \hg^\sigma)\,.$$
\end{proof}

 We recall that a pair $(G,H)$,  where $G$ is a Lie group, and $H\subset G$ a closed subgroup, is called reductive if $\hg$ admits an $\ad_H$-invariant complement in $\g$  \cite[Example 4, p. 165]{SaWa}.  Note that 
 \begin{re}\label{CompRed} Any pair $(G,H)$ with $H$ compact is reductive. In particular, when $G$ is compact, any pair $(G,H)$ with $H\subset G$ a closed subgroup, is reductive. 
 \end{re}

Let $\sigma$ be a infinitesimally homogeneous section with respect to $B_0$. The following result shows that, assuming that $(G,H^\sigma_{q_0})$ is reductive, any connection $B'\in {\cal A}(P^\sigma_{q_0})$  can be modified,   such that the obtained connection $B$  satisfies       
$$(\nabla^{B}_{\MM}\otimes \nabla^B_{\add})(B-B_0)=0,$$
which is   a much stronger property than Proposition \ref{new-prop} (2).  In order to prove this note first that
\begin{re}\label{subbundle-kg}
Suppose that $(G,H^\sigma_{q_0})$ is reductive, and let $\kg$ be an $\ad_{H^\sigma_{q_0}	}$-invariant complement of $\hg^\sigma_{q_0}$ in $\g$.	The direct sum decomposition 
$$\g=\hg^\sigma_{q_0}\oplus\kg\,,
$$
is $\ad_{H^\sigma_{q_0}}$-invariant.  Putting $\kg^\sigma\coloneqq P^\sigma_{q_0}\times_{H^\sigma_{q_0}}\kg$ we obtain a vector bundle decomposition 
$$\ad(Q)=Q\times_G\g=P^\sigma_{q_0}\times_{H^\sigma_{q_0}}\g  
=\hg^\sigma\oplus \kg^\sigma,
$$
which is parallel with respect to any connection on $P^\sigma_{q_0}$ (so to any connection on $Q$ which is compatible with $P^\sigma_{q_0}$).  

Note that the Lie algebra isomorphism $\theta_\g: \ad(Q)_{x_1}\to   \ad(Q)_{x_2}$ associated with  a $G$-isomorphism $\theta:Q_{x_1}\to Q_{x_2}$ satisfying (\ref{LocHomCond1})  maps isomorphically $\kg^\sigma_{x_1}$ onto $\kg^\sigma_{x_2}$. 
\end{re}

  \begin{pr}\label{ConnectionB-Pr}
Suppose that $\sigma$ is  infinitesimally 	homogeneous  with respect to  $B_0$, and the pair $(G,H^\sigma_{q_0})$ is reductive. Then there exists a   connection $B\in {\cal A}(P^\sigma_{q_0})$ such that
 \begin{equation} \label{B-B0-new}  (\nabla^{B}_{\MM}\otimes\nabla^B_\add)(B-B_0)=0\,.
  \end{equation}   
\end{pr}

\begin{proof} Let $B'\in {\cal A}(P^\sigma_{q_0})$, so that Proposition \ref{new-prop} applies.   The problem is to find $\beta \in \Gamma(\Lambda^1_M\otimes \hg^\sigma)$ such that the equation \ref{B-B0-new} holds for $B=B'+\beta$.

We have the following direct sum decompositions
\begin{align}
\ad(Q)&=&{\hg}^\sigma\oplus \kg^\sigma \label{decAd},\\
 \Lambda^1\otimes \ad(Q)&=&(\Lambda^1_M\otimes {\hg}^\sigma)\oplus (\Lambda^1_M\otimes \kg^\sigma)\,. \label{decLambdaAd}	
\end{align}
 Using the splitting $(\ref{decLambdaAd})$ we can decompose $B'-B_0\in  \Lambda_M^1\otimes \ad(Q)$ in a unique way as follows
$$B'-B_0=b_{\hg}+b_{\kg}\,.
$$
where $b_{\hg}\in \Lambda^1_M\otimes {\hg}^\sigma$ and $b_{\kg}\in \Lambda^1_M\otimes {\kg}^\sigma$. 

Put $B\coloneqq B'-b_{\hg}=B_0+b_{\kg}$,  $\beta\coloneqq -b_{\hg}\in \Lambda^1_M\otimes {\hg}^\sigma$. By Remark \ref{subbundle-kg}  the decomposition (\ref{decAd}) is parallel with respect to  $B$.
 Similarly,  the decomposition (\ref{decLambdaAd}) will be $(\nabla^{B}_{\MM}\otimes\nabla^B_\add)$-parallel, so $\Lambda^1_M \otimes {\kg}^\sigma$ is  a $(\nabla^{B}_{\MM}\otimes\nabla^B_\add)$-parallel sub-bundle of $\Lambda^1_M \otimes\ad(Q)$. 
 Since $b_{\kg}$ is a section of $\Lambda^1_M \otimes {\kg}^\sigma$ we obtain 
$$(\nabla^{B}_{\MM}\otimes\nabla^B_\add)_\xi b_{\kg} \in \Gamma(\Lambda^1_M\otimes\Lambda^1_M\otimes {\kg}^\sigma)\ \ \forall \xi\in T_M\,.
$$
On the other hand, using Proposition \ref{new-prop}, we obtain 
$$ (\nabla^{B}_{\MM}\otimes\nabla^B_\add)_\xi b_{\kg} \in \Gamma(\Lambda^1_M\otimes\Lambda^1_M \otimes {\hg}^\sigma)\ \ \forall \xi\in T_M\,.
$$
 Therefore $(\nabla^{B}_{\MM}\otimes\nabla^B_\add)_\xi b_{\kg}=0$, which proves (\ref{B-B0-new}) because $b_{\kg}=B-B_0$.

\end{proof}

Combining Proposition \ref{ConnectionB-Pr} with Remark \ref{parallel-comp-conn}, we obtain

\begin{thry}\label{ConnectionB-Th}
Suppose that $\sigma$ is  infinitesimally 	homogeneous  with respect to  $B_0$. Fix $q_0\in Q$, and suppose that the pair $(G,H^\sigma_{q_0})$ is reductive. There exists a connection $B\in {\cal A}(Q)$ such that:
\begin{enumerate}
\item $((\nabla^{B}_\MM)^{\otimes k}\otimes \nabla^{B}_\EE) (\sigma^{(k)}_{B_0})=0\ \hbox{ for } \ 0\leq k\leq k^\sigma+1$.
\item 	$(\nabla^{B}_{\MM}\otimes\nabla^B_\add)(B-B_0)=0$.
\end{enumerate}
	
\end{thry}

A connection $B\in {\cal A}(Q)$ satisfying the conclusion of Theorem \ref{ConnectionB-Th} will be called an adapted, or Ambrose-Singer type connection for the infinitesimally homogeneous section $\sigma$.  In particular if, $k=0$ then $\sigma^{(0)}_{B_0}=\sigma$, and we obtain the following result:
\begin{co}\label{ConnectionBCo}
	Suppose that $\sigma$ is  infinitesimally 	homogeneous  with respect to  $B_0$. Fix $q_0\in Q$, and suppose that the pair $(G,H^\sigma_{q_0})$ is reductive. Then there exists a   connection $B\in {\cal A}(Q)$ with the properties:
	$$\nabla^{B}_\EE \sigma=0\ , \   (\nabla^{B}_{\MM}\otimes\nabla^B_\add)(B-B_0)=0\,.
	$$  
\end{co}

It is important to note that if $G$ is compact, then the reductivity condition in  Theorem \ref{ConnectionB-Th}, Corollary \ref{ConnectionBCo} will be automatically satisfied.

\section{Applications}\label{ApplicationS}

In this section we will see that our results   provide a  general framework to study locally homogeneous objects by investigating the corresponding infinitesimally homogeneous analogues. For each class of locally homogeneous objects one chooses a system of data $(M,Q,\rho, f:Q\to L(M), B_0,\sigma)$ of the type considered in the previous section, and considers the corresponding infinitesimally homogeneity condition, which a priori is weaker than local homogeneity. In the presence of an infinitesimally homogeneous object our general Theorem \ref{ConnectionB-Th} will yield an adapted connection $B$, which can be used to prove that any infinitesimally homogeneous object is locally homogeneous, and is even globally homogeneous if certain topological and completeness conditions are satisfied.

\subsection{LH Riemannian manifolds, Ambrose-Singer Theorem}

In this section we explain briefly how the Ambrose-Singer theorem (Theorem \ref{ASTheorem} ) can be obtained using our Theorem \ref{ConnectionB-Th}. As explained above we need first to chose an appropriate system of data $(M,Q,\rho, f:Q\to L(M), B_0,\sigma)$.

Let $Q\coloneqq \O(M)$ be the orthonormal frame bundle of $(M,g)$ and  $C_0$ be be the Levi-Civita   connection on it. Let $f:Q\to L(M)$ be the inclusion bundle map and $\sigma\coloneqq R^g $ be the Riemann curvature tensor of $g$, regarded as a section of the vector bundle $E\coloneqq (\Lambda^1_M)^{\otimes4}\simeq Q\times_\rho(\R^{n*})^{\otimes 4}$. It is easy to see that   Singer's  infinitesimally homogeneous condition for $g$ \cite{Si} is  equivalent to our infinitesimally homogeneous condition for the section $R^g\in\Gamma(E)$ with respect to $C_0$ (see Definition \ref{InfHomSectDef}). 

Suppose that the Riemannian curvature tensor $R^g$ is infinitesimally homogeneous with respect to $C_0$. By Corollary \ref{ConnectionBCo} there exists a connection $C\in \cal {A}(\O(M))$ 
such that:
$$\nabla^{C}_\EE R^g=0\ , \   (\nabla^{C}_{\MM}\otimes\nabla^C_\add)(C-C_0)=0\,.
	$$  
Let $\nabla$ be the metric connection on $M$ induced by $C$; then $\nabla R^g=0$, $\nabla (C-C_0)=0$ and Corollary \ref{C_C_0CO} below shows that $\nabla$ is indeed an Ambrose-Singer connection.

%
\begin{pr}\label{C_C_0Pr} Let  $\nabla_0$ ( resp. $\nabla$) be the linear connections on smooth manifold $M$ associated with $C_0\in \cal {A}(L(M))$ (resp.  $C\in \cal {A}(L(M))$). Let 
$$S\coloneqq \nabla-\nabla_0=C-C_0\in A^1(\End(T_M))=A^1(\ad(L(M)).$$
 Then the following conditions are equivalent:
\begin{enumerate}
\item $\nabla R^{\nabla_0}=\nabla T^{\nabla_0}=\nabla S=0\,,$
\item $\nabla R^{\nabla}=\nabla T^{\nabla}=\nabla S=0.$
\end{enumerate}
\end{pr}
\begin{proof}
	The difference $T^\nabla- T^{\nabla_0}$ 
	 is   the image of $C-C_0$ under the bundle  morohism  $\Lambda^1\End(T_M)\to L^2_{\mathrm{alt}}(T_M,T_M)$
	given by  $S\mapsto T_S$, where
	$$T_S(X,Y)\coloneqq S(X)(Y)-S(Y)(X)\,.
	$$
	This  morphism is induced by an $\GL(n)$-equivariant isomorphism 
	$$\R^{n*}\otimes\gl(n)\to L^2_\alt(\R^n,\R^n),$$
	so it is parallel with respect to any linear connection on $M$. Therefore   $\nabla S =0$ implies $\nabla T_S=0$, so, under the assumption $\nabla S=0$, the conditions $\nabla T^{\nabla_0}=0$, $\nabla T^{\nabla}=0$ are equivalent.
	\vspace{2mm}\\
	
	Since $\nabla_0=\nabla-S$, we have $R^{\nabla_0}=R^{\nabla}-d^\nabla S+\frac{1}{2}[S\wedge S]$. A direct computation shows that the assumption $\nabla S=0$ implies
	\begin{enumerate}
		\item $\nabla[S\wedge S]=0$.
		\item $(d^\nabla S)(X,Y)=S(T^\nabla(X,Y))$	 for any vector fields $X$, $Y\in \mathfrak {X}(M)$.
	\end{enumerate}
	
	Therefore, under the assumptions $\nabla S=0$, $\nabla T^\nabla=0$, the conditions $\nabla R^{\nabla_0}=0$ and $\nabla R^{\nabla}=0$ are equivalent. %
\end{proof}

Suppose now that $C_0$ is the Levi-Civita connection of a Riemannian manifold $(M,g)$, and let $C\in {\cal A}(\O(M))$. The morphism $ \Lambda^1\so(T_M)\to L^2_{\mathrm{alt}}(T_M,T_M)$ given by $S\mapsto T_S$ is a bundle isomorphism, which is   $\O(n)$-equivariant, hence parallel with respect to any metric connection on $M$. Therefore, in this case the conditions $\nabla T^\nabla=0$ and $\nabla (C-C_0)=0$ are equivalent, so Proposition \ref{C_C_0Pr}  gives

\begin{co}\label{C_C_0CO} Let  $C_0\in \cal {A}(\O(M))$ be the Levi-Civita connection of $(M,g)$, and $R^g $ be the Riemann curvature tensor. Let $\nabla$ be a linear metric connection on $M$ associated with $C\in \cal {A}(\O(M))$. The following  conditions are equivalent.
\begin{enumerate}
\item $\nabla R^g=\nabla (C-C_0)=0\,,$
\item $\nabla R^{\nabla}=\nabla T^{\nabla}=0.$
\end{enumerate}
\end{co}
%

Note that the existence of an Ambrose-Singer connection has a fundamental consequence, namely Singer's theorem which states that any infinitesimally homogeneous (in particular any locally homogeneous) complete and simply connected Riemannian manifold is globally homogeneous \cite{Si}. This result holds in the differentiable framework (does not need the real analyticity). More precisely one can prove the next theorem by applying the \cite[Corollary 7.9]{KN} to an Ambrose-Singer connection. 
\begin{thry}
A Riemannian manifold $(M,g)$ is infinitesimally homogeneous if and only if it is locally homogeneous. Any connected, simply connected, complete infinitesimally homogeneous Riemannian manifold is homogeneous.
\end{thry}

 The Ambrose-Singer theorem has been extended to the general framework of locally homogeneous pseudo-Riemannian manifolds. Since the structure group of the frame bundle is not necessary compact,   one needs an additional reductivity condition \cite{GO,Lu,CL}. Similar generalizations   can also be obtained using  Theorem \ref{ConnectionB-Th}.

  \subsection{\texorpdfstring{Kiri\v{c}enko's}{Ki} theorem for LH systems \texorpdfstring{$(g,P_1,\dots,P_k)$}{STr1}. }

 Let $M$ be a smooth manifold of dimension $n$. 
 \begin{dt} A system $(g,P_1,\dots,P_k)$ consisting of a Riemannian metric $g$ and tensor fields $P_1$, \dots, $P_k$ on $M$ is called locally homogeneous if for any two points $x_1$, $x_2\in M$ there is a an isometry $\varphi:U_1\to U_2$ between open neighborhoods $U_i\ni x_i$ such that $\varphi(x_1)=x_2$ and $\varphi_*(\resto{P_j}{U_1})=\resto{P_j}{U_2}$.
 \end{dt}	
 
 A locally homogeneous system $(g,P_1,\dots,P_k)$ defines a pseudogroup of local isometries of $(M,g)$ which acts transitively on $M$. So, in the terminology used  by Kiri\v{c}enko \cite{Ki}, one obtains a ``geometric structure" on $M$ which is associated with the metric $g$ and the system $(P_1,\dots,P_k)$.

Each tensor field $P_j$ is a section in  the vector bundle ${\cal T}^{r_j}_{s_j}(M)\coloneqq T_M^{\otimes r_j}\otimes (\Lambda^1_M)^{\otimes s_j}$.   The following result due to Kiri\v{c}enko \cite{Ki}, extends the Ambrose-Singer theorem to manifolds with a geometric structure defined by a locally homogeneous system $(g,P_1,\dots,P_k)$.
\begin{thry}\label{KiriTh}
A system  $(g,P_1,\dots,P_k)$  consisting of a Riemannian metric $g$ and tensor fields $P_1$, \dots, $P_k$ on $M$ is	 locally homogeneous  if and only if there exists a metric connection $\nabla$ such that
\begin{equation}\label{KiriEq}\nabla R^{\nabla}=\nabla T^{\nabla}=\nabla P_1=\cdots=\nabla P_k=0.
\end{equation}
If $M$ is simply connected, then any  locally homogeneous system  $(g,P_1,\dots,P_k)$ with $(M,g)$ complete is globally homogeneous in the following sense: the exists a  transitive group of isometries of $(M,g)$ leaving the tensor fields $P_j$ invariant. 
\end{thry}

In the presence of locally homogeneous system $(g,P_1,\dots,P_k)$, a metric connection satisfying (\ref{KiriEq}) is called an Ambrose-Singer-Kiri\v{c}enko connection  for the system $(g,P_1,\dots,P_k)$ (or for the geometric structure defined by this system) \cite{Lu}.
\vspace{2mm}

Let $Q\coloneqq \O(M)$ be the orthonormal frame bundle of $(M,g)$,  $C_0$ be its Levi-Civita   connection, and $f:Q\to L(M)$ be the obvious inclusion bundle map. The system  $(R^g,P_1,\cdots,P_k)$ defines a section  $\sigma$  of the vector bundle
$$E\coloneqq (\Lambda^1_M)^{\otimes4}\oplus {\cal T}^{r_1}_{s_1}(M) \oplus {\cal T}^{r_2}_{s_2}(M)\oplus \cdots \oplus {\cal T}^{r_k}_{s_k}(M)\,.$$
Since the system $(g,P_1,\dots,P_k)$ is locally homogeneous, the section $\sigma\in \Gamma(E)$ is infinitesimally homogeneous with respect to $C_0$ (see Definition \ref{InfHomSectDef}). Using Corollary \ref{ConnectionBCo} there exists a connection $C\in \cal {A}(\O(M))$ 
such that:
$$   \nabla^{C}_\EE \sigma=0\ \hbox{,}\   (\nabla^{C}_{\MM}\otimes\nabla^C_\add)(C-C_0)=0\,.$$
Let $\nabla$ be the metric connection on $M$ induced by $C$. Then, one obtains 
$$\nabla R^g=0 , \nabla (C-C_0)=0,\ \nabla P_1=\nabla P_2=\cdots=\nabla P_k=0.$$
By Corollary \ref{C_C_0CO} we see that $\nabla$ is indeed an Ambrose-Singer-Kiri\v{c}enko connection. The converse implication is proved by applying the \cite[Corollary 7.5]{KN} (and its proof method) to a connection satisfying (\ref{KiriEq}). More precisely, for two points $x_1$, $x_2\in M$ we choose a smooth path $\gamma:[0,1]\to M$ with $\gamma(0)=x_1$, $\gamma(1)=x_2$. The parallel transport with respect to $\nabla$ defines an isometric isomorphism $F:T_{x_1}M\to  T_{x_2}M$ mapping $R^\nabla_{x_1}$, $T^\nabla_{x_1}$, $P_{jx_1}$ to $R^\nabla_{x_2}$, $T^\nabla_{x_2}$, $P_{jx_2}$  respectively. As in \cite[Corollary 7.5]{KN} $F$ gives a local $\nabla$-affine isomorphism $\varphi$, which will be isometric and leave the tensor fields $P_j$ invariant. 

In the  complete, simply connected case the $\nabla$-affine isomorphisms $\varphi$ obtained in this way extend to global isometries of $(M,g)$.
\\

Note that a generalization of Theorem \ref{KiriTh} to pseudo-Riemannian manifolds can be found in  \cite{GO}, \cite{Lu}. 
\begin{ex} The case of locally homogeneous almost Hermitian manifolds is considered in \cite{Se}. An almost Hermitian manifold $(M,g,J)$ is locally homogeneous if, for every two points $x_1$, $x_2\in M$ there is a Hermitian isometry $\varphi:U_1\to U_2$ between open neighborhoods $U_i\ni x_i$ sending $x_1$ to $x_2$. By a Hermitian isometry we mean a diffeomorphism which is compatible with both the almost complex structure and the Hermitian metric.

A characterization theorem similar to Theorem \ref{ASTheorem}  was proved by Sekigawa \cite{Se} for almost Hermitian manifolds. The following local version can be obtained using Kiri\v{c}enko's theorem by choosing $k=1$ and $P_1=J$.

\begin{thry} \label{AS-Local-verionHer}\cite{Se}, \cite[Theorem 1]{CN}  An almost Hermitian  manifold $(M,g,J)$ is locally homogeneous if and only if there exists a metric connection $\nabla$ such that
$$\nabla R^\nabla=\nabla T^\nabla=\nabla J=0.
$$
If $M$  is simply connected, then any complete locally  homogeneous almost Hermitian structure on $M$ is (globally) homogeneous.
\end{thry}

  \end{ex} 
  
\subsection{Opozda's theorem on locally homogeneous \texorpdfstring{$G$}{Str2}-structures}

 In \cite{O2, O3} affine connections and their affine transformations has been investigated by Opozda and in \cite{O1} the notion of infinitesimal homogeneity is extended to arbitrary connections on $G$-structures. In this section we show that Opozda's theorem can be proved using Theorem \ref{ConnectionB-Th}. Our Lemma \ref{CoAnnConn2} will allow us to show that the analyticity condition required in Opozda's statement is not necessary.

 Let $M$ be an $n$-dimensional manifold, and $G$ be a Lie subgroup of the linear group $\GL(n, \R)$. We recall that a     $G$-structure on   $M$, is a  sub-bundle   $P\subset L(M)$ where $P$ is a principal $G$-bundle, and the inclusion map is $G$-equivariant. 
 \begin{dt}
 Let  $P\subset L(M)$ be a 	$G$-structure on   $M$. A connection $C_0$ on $P$ is called
 locally homogeneous, if for any two points $x_1$, $x_2\in M$ there exists a  diffeomorphism $x_1\in  U_1\textmap{f} U_2\ni x_2$ between open neighborhoods of $x_i$, such that $df(P_{U_1})\subset P_{U_2}$, and $f$ is a $\nabla_0$-affine isomorphism, where  $\nabla_0$ is the linear connection associated with $C_0$. 
 \end{dt}

Opozda also introduces a natural infinitesimal homogeneity condition for connections on the principal bundle of a $G$-structure (see \cite[Definitions 1.3,   1.5]{O1}). Using our formalism (see Definition \ref {InfHomSectDef}), this condition is equivalent to
\begin{dt}
 Let  $P\subset L(M)$ be a 	$G$-structure on   $M$. A connection $C_0$ on $P$ is called
 infinitesimally homogeneous, if	 the pair $(T^{\nabla_0},R^{\nabla_0})$, regarded as a section in the associated vector bundle 
 $$E\coloneqq P\times_G(L^2_\alt(\R^n,\R^n)\oplus L^2_\alt(\R^n,\End(\R^n)),$$
   is  infinitesimally homogeneous with respect to $C_0$.
\end{dt}

Suppose that $C_0$ is  infinitesimally homogeneous, i.e. the section $\sigma_0=(T^{\nabla_0},R^{\nabla_0})$ is infinitesimally homogeneous with respect to $C_0$. By Theorem \ref{ConnectionB-Th}  we obtain  a connection $C\in \cal {A}(P)$ 
	such that for $0\leq k\leq k^{\sigma_0}+1$:
	\begin{equation} \label{ConnCOpozda}
	   ((\nabla^{C}_\MM)^{\otimes k}\otimes \nabla^{C}_\EE) (\sigma^{(k)}_{C_0})=0\ \hbox{,}\    (\nabla^{C}_{\MM}\otimes\nabla^C_\add)(C-C_0)=0\,.
	 \end{equation}
	Let $\nabla$ denote the linear connection on $M$ induced by $C$. Taking $k=0$, one obtains 
	$$\nabla T^{\nabla_0}=0 \ , \nabla R^{\nabla_0}=0 \ , \nabla (C-C_0)=0.$$
	 Proposition \ref{C_C_0Pr} with $S\coloneqq \nabla-\nabla_0$ gives
	$$\nabla R^{\nabla}=\nabla T^{\nabla}=\nabla S=0.
	$$

  Lemma \ref {CoAnnConn2} proved below shows that the atlas consisting of  $\nabla$-normal coordinate systems is analytic, and with respect to the corresponding analytic structure, the connection $\nabla_0=\nabla -S$ is analytic. The equation (\ref {ConnCOpozda}) implies that for $0\leq k\leq k^{\sigma_0}+1$:
$$\nabla (\nabla_0^{(k)}T^{\nabla_0})=0\ , \nabla (\nabla_0^{(k)}R^{\nabla_0})=0.
$$
Using \cite[Lemma 2.1]{O1} the above equations are verified for all $k\geq0$. 
 
Fot two points $x_1$, $x_2\in M$ we choose a smooth path $\gamma:[0,1]\to M$ with $\gamma(0)=x_1$, $\gamma(1)=x_2$. The parallel transport with respect to $\nabla$ defines a linear isomorphism $F:T_{x_1}M\to  T_{x_2}M$ mapping $ (\nabla_0^{(k)}T^{\nabla_0})_{x_1}$, $(\nabla_0^{(k)}R^{\nabla_0})_{x_1}$ to $ (\nabla_0^{(k)}T^{\nabla_0})_{x_2}$, $(\nabla_0^{(k)}R^{\nabla_0})_{x_2}$  respectively.  Using \cite[Theorem 7.2]{KN} $F$ defines a local $\nabla_0$-affine isomorphism $\varphi$ sending $x_1$ to $x_2$.

Therefore, we obtain  Opozda's theorem (\cite[Theorem 2.2 ]{O1}):
\begin{thry}\label{Opzda}
Let $P\subset L(M)$ be a $G$-structure on a manifold $M$ and $C_0$ be an infinitesimally $P$-homogeneous connection on $P$. Fix $q_0\in P$, and suppose that the pair $(G,H^{\sigma_0}_{q_0})$ is reductive. Then $\nabla_0$ is locally homogeneous.
\end{thry}
 
\begin{lm}
Let $M$ be an analytic $n$-manifold, $G$ be a Lie group, $p:P\to M$ be an analytic principal $G$-bundle 	on $M$, and $A\in {\cal A}(P)$ be an analytic connection on $P$. Let $\rho:G\to \GL(F)$ be a representation of $G$ on an $r$-dimensional vector space $F$, $E\coloneqq P\times_\rho F$ be the associated vector bundle, and $\nabla^\rho_A$ be the linear connection on $E$ associated with $A$. Then
\begin{enumerate}
\item Let $\sigma\in A^0(M,E)$ such that $\nabla^\rho_A\sigma$ is an analytic $E$-valued 1-form. Then $\sigma$ is analytic.	
\item Let $\nabla$ be an analytic connection on $M$, and $\alpha\in A^1(M,\ad(P))$ such that the derivative  $(\nabla\otimes\nabla^A_\ad)(\alpha)$ vanishes in $A^0(\Lambda_M^1\otimes \Lambda_M^1\otimes\ad(P))$. Then $\alpha$ is analytic, so $B\coloneqq A+\alpha$ is an analytic connection on $P$.
\end{enumerate}
\end{lm}
\begin{proof} (1) The problem is local, so it sufficient to prove that $\alpha$ is analytic around any point $x\in M$.  With respect to 
\begin{itemize}
\item 	an analytic chart $h:U\to V\subset\R^n$, 
\item    an analytic trivialization $\tau:P_U\to U\times G$ of $P$,
\item a basis $b$ of $F$,
\end{itemize}
the restriction of the operator  $\nabla^\rho_A$ to  a sufficiently small open neighborhood $U$ of $x$  can be identified  with a differential  operator of the form
$$d+ a_{h,\tau,b}: A^0(V,\R^r)\to A^1(V,\R^r)\,,
$$
where $a_{h,\tau,b}\in A^1(V,\gl(r))$ is analytic. Let $\sigma_{h,\tau,b}\in A^0(V,\R^r)$ be the $\R^r$-valued map associated with $\sigma$. The hypothesis implies that $(d+ a_{h,\tau,b})\sigma_{h,\tau,B}$ is analytic, hence
$$d^*(d+ a_{h,\tau,b})\sigma_{h,\tau,b}
$$
is also analytic. But $d^*(d+ a_{h,\tau,b})$ is an elliptic operator with analytic coefficients, so, by analytic elliptic regularity \cite[Theorem 40, p. 467]{Be}, it follows that $\sigma_{h,\tau,b}$ is analytic. Therefore $\resto{\sigma}{U}$ is analytic.
\\ \\
(2) Let $C\in {\cal A}(L(M))$ be the connection on the frame bundle $L(M)$ which corresponds to $\nabla$. The statement follows by applying  (1) to the connection defined by the pair $(C,A)$ on the product bundle $L(M)\times_M P$ and the representation $\GL(\R)\times G\to \GL(\R^n\otimes \g)$ given by $(u,l)\mapsto (u^t)^{-1}\otimes \ad_l$. 
	
\end{proof}

\begin{lm}\label{CoAnnConn1}
Let $M$ be an analytic manifold, and $\nabla$ be an analytic connection on $M$. Let $S\in A^1(\End(T_M))$ be such that $\nabla S$ is analytic. Then $S$ is analytic, so the connection $\nabla+ S$ is also analytic. 	
\end{lm}

\begin{lm}\label{CoAnnConn2}
Let $M$ be a differentiable manifold,  let $\nabla$ be a linear connection on $M$ such that $\nabla T^\nabla=0$, 	$\nabla R^\nabla=0$, and let $S\in A^1(\End(T_M))$ be such that $\nabla S=0$. The atlas consisting of  $\nabla$-normal coordinate systems is analytic, and with respect to the corresponding analytic structure, the connections $\nabla$ and $\nabla +S$ are analytic.
\end{lm}
\begin{proof}
By \cite[Theorem 7.7]{KN} the atlas consisting of  $\nabla$-normal coordinate systems is analytic and with respect to the corresponding analytic structure the connection $\nabla$ is analytic. By Lemma \ref{CoAnnConn1} the connection $\nabla+S$ is analytic.	 
\end{proof}

\subsection{Locally homogeneous spinors}

We have seen that the case of LH  almost Hermitian structures can be obtained as a special case of Kiri\v{c}enko's formalism dedicated to LH tensors. We give now an example which cannot be obtained as a special case of this formalism. In this example we have to consider a section in a vector bundle which is not associated with the frame bundle of the base manifold. Let $r: \Spin(n)\to \SO(n)$ be the canonical epimorphism, and $(M,g)$ be an oriented Riemannian $n$-manifold. A Spin structure of $(M,g)$ is a bundle morphism $\Lambda:Q\to \SO(M)$ of type $r$,  where $Q$ is a $\Spin(n)$-bundle. The short exact sequence
$$
1\to \{\pm1\}\to \Spin(n)\to \SO(n)\to 1
$$
shows  $M$ is a Spin manfold (i.e. it admits a Spin structure) if and only if its Stiefel-Whitney class $w_2(M)$ vanishes, and, if non-empty, the set of isomorphism classes of Spin structures on $M$ is an $H^1(M,\Z_2)$-torsor. In particular a simply connected Spin manifold admits an (up to isomorphism) unique Spin structure.

Let     $\kappa:\Spin(n)\to \mathrm{\GL}(\Delta_n)$ be the spin representation and $S\coloneqq Q\times_{\kappa}{\Delta_n}$ be the associated spinor bundle, where $\Delta_n$ is the vector space of Dirac spinors defined by 
$$\Delta_n=\underset{k \ \hbox{\small{times}}}{\underbrace{\C^2\otimes\dots\otimes\C^2}} \ \ \ \ \hbox{for} \ \ n=2k,\ 2k+1.
$$

Recall that, if $n$ is even, then $S$ comes with a Levi-Civita parallel orthogonal decomposition $S=S^+\oplus S^-$(see \cite[Section 2.5]{Fr}). Since the Lie group morphism $r$ is a local isomorphism it follows that a Spin structure $\Lambda$ induces a bijection ${\cal A}(Q)\to{\cal A}(\SO(M))$ between the spaces of connections on the two bundles. Denoting by $\tilde C_0\in {\cal A}(Q)$ the lift of the Levi-Civita connection $C_0$ to  $Q$, we see that, if $(M,g)$ is locally homogeneous, then $(g,Q,\tilde C_0)$ is a locally homogeneous triple in the sense of Definition \ref{DefSymLH}. Indeed, if $\varphi:U_1\to U_2$ is an isometry between simply connected sets $U_1$, $U_2\subset M$, then, by Lemma \ref{Lifts} below, the bundle map $\tilde\varphi_*:\SO(U_1)\to\SO(U_2)$ induced by the tangent map of $\varphi$ lifts to a bundle map $\Psi:Q_{U_1}\to Q_{U_2}$. Since $\tilde\varphi_*$ leaves the Levi-Civita connection $C_0$ invariant, it follows that $\Psi$ leaves the  connection $\tilde C_0$ invariant.

\begin{dt}\label{LHspinor}
Let $(M,g)$ be a LH Riemannian manifold. A spinor $s\in \Gamma(M,S)$ will be  called LH if for any pair ($x_1,x_2)\in M\times M$ there exists an  isometry  $ U_1\textmap{\varphi} U_2$  between open neighborhoods $U_i\ni x_i$, and a $\tilde\varphi_*$-covering bundle isomorphism $ Q_{U_1}\textmap{\Psi}Q_{U_2}$ such that $\Psi_*(\resto{s}{U_1})=\resto{s}{U_2}$.
\end{dt}

We also have a natural infinitesimal  homogeneity condition: this is just the infinitesimal  homogeneity condition for $s$ regarded as a section of $S$ with respect to the lift  $\tilde{C_0}$ of the Levi-Civita connection $C_0$ to  $Q$. This condition is apparently weaker than the LH condition. 
\begin{re} Let $(M,g)$ be a spin LH Riemannian manifold, and 
 $s$ a LH spinor, then the section $(R^g,s)$ is infinitesimally homogeneous with respect to connection $\tilde{C_0}\in {\cal A}(Q)$.	
\end{re}

\begin{proof} Let $\varphi:U_1\to U_2$ , $\Psi:Q_{U_1}\to Q_{U_2}$ as in  Definition \ref{LHspinor}. Since  $\Psi$   lifts  $\tilde \varphi_*$,  the induced vector  bundle map $T_{U_1}\to T_{U_2}$ is precisely the tangent map $\varphi_*$. Taking also into account that $\Psi$  leaves the connection $\tilde C_0$ invariant (because $\tilde\varphi_*$ leaves $C_0$ invariant) and maps $s|_{U_1}$ onto $s|_{U_2}$, it follows that it maps $(R^g)_{\tilde C_0}^{(i)}|_{U_1}$ onto $(R^g)_{\tilde C_0}^{(i)}|_{U_2}$  and $s_{\tilde C_0}^{(i)}|_{U_1}$ onto $s_{\tilde C_0}^{(i)}|_{U_2}$ for any $i$.\end{proof}

A metric connection $\nabla$ on $M$ defines a connection on $\SO(M)$, so, via the identification ${\cal A}(\SO(M))={\cal A}(Q)$, it defines a Hermitian connection (which will be denoted by the same symbol $\nabla$) on  the spinor bundle $S$.
\begin{thry} \label{ASspinor}
	Let  $s\in \Gamma(M,S)$ be a locally homogeneous spinor. There exists an Ambrose\,-\,Singer connection $\nabla$ on $(M,g)$ which leaves $s$ invariant. 
\end{thry}
\begin{proof}
 Put $G\coloneqq \Spin(n)$.  Consider the associated vector bundle  $E\coloneqq (\Lambda^1_M)^{\otimes4} \oplus S$, associated with the representation $\rho:G\to \GL((\R^{n*})^{\otimes 4}\oplus \Delta_n)$ induced by $(r,\kappa)$.
 
   Let  $R^g$ be the Riemann curvature tensor and $s\in \Gamma(M,S)$ be a spinor. The pair $(R^g,s)$ defines a section $\sigma$ of vector bundle $E$, which is obviously infinitesimally homogeneous with respect to $ \tilde C_0$. Corollary \ref{ConnectionBCo} gives a connection $\tilde C \in {\cal A}(Q)$	 such that 
$$    \nabla^{\tilde C}_\EE \sigma =0\ ,\ (\nabla^{\tilde C}_\MM\otimes \nabla^{\tilde C}_\ad)(\tilde C-\tilde C_0)=0\,.
$$
Using the identification ${\cal A}(\SO(M))={\cal A}(Q)$, we  denote by the same symbol $\nabla$   the connections on the vector bundles $T_M$, $S$ induced by $\tilde C$. The above equations are equivalent to 
$$\nabla R^g =0,\  \nabla  (\tilde C-\tilde C_0)=0,     \nabla s=0\,.$$
By Corollary \ref{C_C_0CO}, we see that $\nabla$ is an Ambrose-Singer connection which leaves the spinor $s$ invariant.

\end{proof}

\begin{co}\label{spin-coro}
Let $(M,g)$ be a  simply connected, complete LH Riemannian manifold, and $s\in \Gamma(M,S)$ is a locally homogeneous spinor. There exists a Lie group $G$ and an action $\beta:G\times Q\to Q$ by bundle isomorphisms with the following properties:
\begin{enumerate}
	
\item $\beta$ lifts  a transitive action by isometries   $\alpha:G\times M\to M$.
\item  $\beta$ leaves $s$ invariant. 
\end{enumerate}
\end{co}
\begin{proof}

Using Theorem \ref{ASspinor}, let $\nabla$ be an Ambrose-Singer connection on $(M,g)$ which leaves $s$ invariant. Let $\Gamma$ be the Lie group of isometric $\nabla$-affine transformations of $M$.  The proof of the classical  Ambrose-Singer theorem shows that $\Gamma$ acts transitively on $M$. By Lemma \ref{Lifts}  below, the group of lifts
$$
\Gamma_Q\coloneqq \{\Psi:Q\to Q\hbox{ bundle isomorphisms} |\ \exists\varphi\in\Gamma\hbox{ such that }\Psi \hbox{ lifts }\tilde\varphi_*\}
$$
is a trivial extension of $\Gamma$ by $\{\pm1\}$. The group 
$$
G\coloneqq \{\Psi\in\Gamma_Q|\ \Psi(s)=s\}
$$
is a closed subgroup of $\Gamma_Q$, so it has a natural Lie group structure. We will prove that $G$ acts transitively on $M$ via the natural morphism $G\to \Gamma$. Let $(u,v)\in M\times M$.  

The existence of an element $\varphi\in \Gamma$ such that $\varphi(u)=v$ is proved as follows  (see \cite[Corollary 7.5 p  262]{KN}):   choose a smooth path $\gamma:[0,1]\to M$  such that $\gamma(0)=u$, $\gamma(1)=v$. Parallel transport along $\gamma$ with respect to $\nabla$ gives a linear isometry $f:T_uM\to T_vM$ mapping $T^\nabla_u$ onto $T^\nabla_v$ and $R^\nabla_u$ onto $R^\nabla_v$. Since $g$ is complete and $\nabla$ is compatible with $g$, it follows  that $\nabla$ is complete (see \cite[Proposition 1.5]{TV}) so, by \cite[Corollary 7.9 p  265]{KN}  there exists  a unique $\nabla$-affine transformation $\varphi:M\to M$ whose tangent map at $u$ is $f$. Since $\varphi_{*u}=f$ is an isometry, and $\varphi$ is affine with respect to the metric connection $\nabla$, it follows that $\varphi$ is an isometry.

 Let $C$ be the connection on $\SO(M)$ associated with $\nabla$, and $\tilde C$ be its lift to the $\Spin(n)$-bundle $Q$.  Note that any lift $\Psi$  of $\tilde\varphi_*$  leaves the connection $\tilde C$ invariant. Indeed,   $\tilde\varphi_*$   leaves $C$ invariant (because $\varphi$ is $\nabla$-affine), so  $\tilde C$ and $\Psi_*(\tilde C)$   induce the same  connection $C$ on $\SO(M)$, so they coincide.

We claim that there exists a lift $\Psi$ of $\tilde\varphi_*$ which leaves the spinor $s$ invariant.   Parallel transport along the same curve $\gamma$ with respect to $\tilde C$ gives  a $\Spin(n)$-equivariant diffeomorphism $ Q_u\textmap{\theta} Q_v$  which lifts $\tilde\varphi_{*u}$ and (since $\nabla$ leaves $s$ invariant) maps $s_u$ onto $s_v$. Let $\Psi$ be the unique lift of $\tilde\varphi$ such that $\Psi_u=\theta$ (see the proof of Lemma \ref{Lifts}). The spinors $s$ and $\Psi(s)$ are both $\tilde C$-parallel and coincide at $v$, so they coincide. By construction we have  $\Psi\in G$. Since $\Psi$ lifts $\varphi$, it maps $u$ onto $v$, which completes the proof.
\end{proof}

\begin{lm}\label{Lifts}
Let $M_1$, $M_2$ be a connected differentiable manifolds, $P_i$ be an $\SO(n)$-bundle on $M_i$, and $\Lambda_i:Q_i\to P_i$ be a bundle morphism of type $r$. Let $\varphi:M_1\to M_2$ be a diffeomorphism, and $\Phi:P_1\to P_2$ be a  $\varphi$-covering bundle isomorphism. If   $H^1(M_1,\Z_2)=0$, then the set of $\Phi$-covering bundle isomorphims $Q_1\to Q_2$ has two elements.
\end{lm}
\begin{proof}
For a point $x\in M_1$ the set $I_x$ of  $\Phi_x$-lifting $\Spin(n)$-equivariant diffeomorphisms $Q_{1,x}\to Q_{2,\varphi(x)}$  is a $\{\pm1\}$-torsor, and the union $\cup_{x\in M}I_x$ has a manifold structure such that the obvious projection $I\to M_1$ is a double cover.   Since $H^1(M_1,\Z_2)=0$, this double cover is trivial, and its  sections give two global bundle isomorphisms 	$Q_1\to Q_2$  lifting $\Phi$.
\end{proof}

\section{Locally homogeneous triples}\label{LHCchap}
\subsection{Infinitesimally homogeneous triples}
Let $(M,g)$ be a connected, compact, locally homogeneous Riemannian manifold, $K$ be a connected, compact Lie group, and $p:P\to M$ be a principal $K$-bundle on $M$. 
Let $A$ be a connection on $P$ and fix an $\ad$-invariant inner product on the Lie algebra $\kg$ of $K$; these data   define a $K$-invariant Riemannian metric $g_A$ on $P$ which makes $p$ a Riemannian submersion. Such metrics on principal bundles are called  {\it connection metrics} and have been studied in the literature \cite{FZ, Je, WZ}.   


The connection metric $g_A$ associated with any LH triple  $(g,P\textmap{p} M,A)$ is locally homogeneous, hence it defines a geometric structure in the sense of Thurston \cite{Th} on the total space $P$.  
Therefore the classification of LH triples on a given  connected,  compact LH Riemannian manifold $(M,g)$  is related to the classification of compact geometric  manifolds  which are principal bundles over a geometric base. 
  We refer to \cite{Ba2} and \cite{BaTe} for classes of examples of geometric manifolds which can be obtained in this way (using a LH triple).
In \cite{BaTe} we proved explicit classifications theorems for LH triples on Riemann surfaces. 
\vspace{1.5mm}

%

Let $(M,g)$ be an $n$-dimensional Riemannian manifold, $K$ be a compact Lie group, $P$ be a principal $K$-bundle on $M$, and $A_0$ be a connection on $P$.  Put 
$$G\coloneqq \O(n)\times K\,,Q\coloneqq \O(M)\times_M P\,,\ V\coloneqq (\R^{n*})^{\otimes 4}\oplus \big((\R^{n*})^{\otimes 2}\otimes \kg\big)\,.$$
Note that $G$ comes with obvious representations $r:G\to\GL(n)$, $\rho:G\to \GL(V)$, and $Q$ comes with an obvious bundle morphism $f:Q\to L(M)$ of type $r$  given by the projection  $\O(M)\times_M P\to \O(M)\subset L(M)$.

 The Riemann curvature tensor $R^g$ of $g$  will be regarded as a section of $(\Lambda^1_M)^{\otimes 4}$, and the curvature $F^{A_0}$ of connection $A_0$ will be regarded as a section of the vector bundle  $\Lambda^2_M\otimes\ad(P)\subset (\Lambda^1_M)^{\otimes 2}\otimes\ad(P)$. The pair $\sigma\coloneqq (R^g,F^{A_0})$ is a section of the associated vector bundle $E\coloneqq Q\times_G V$. Let $C_0$ denote the Levi-Civita connection on $\O(M)$ and let $B_0$ be the connection on $Q$ defined by the pair of connections 
$$B_0\coloneqq (C_0,A_0)\in \cal A(\O(M))\times \cal A(P)={\cal A}(Q).$$
\begin{dt} \label{IHTDef} The triple $(g,P\textmap{p} M,A_0)$ will be called infinitesimally homogeneous if the section $\sigma\coloneqq (R^g,F^{A_0})\in \Gamma(E)$ is infinitesimally homogeneous with respect to $B_0$ in the sense of Definition \ref{InfHomSectDef}.
\end{dt}

This condition can be reformulated explicitly as follows: 
\vspace{1mm}

Let $\nabla^{g}$ denotes the Levi-Civita connection of $g$ on $M$ and, let $\nabla^{A_0}$ denote the associted connection on adjoint bundle $\ad(P)$.  We  denote by $\nabla^{g}\otimes \nabla^{A_0}$ the tensor product connection  on  $\Lambda_M^2\otimes\ad(P)$. More generally, we obtain a tensor product connection $(\nabla^{g}\otimes \nabla^{A_0})^i=(\nabla^{g})^{\otimes  (i-1)}\otimes (\nabla^{g}\otimes \nabla^{A_0})$ on the vector bundle 
$$(\Lambda_M^1)^{\otimes (i-1)}\otimes\Lambda_M^2\otimes\ad(P).
$$
For $x\in M$, the Lie algebra of skew-symmetric endomorphism of the Euclidian space $(T_xM,g_x)$ will be denoted by $\so(T_xM)$.
    For any $k\in\N$ and $x\in M$ we define $\hg^{g,A_0}_x{(k)}$ to be the set of all pairs $(u,v)\in \so(T_xM)\oplus\ad({P}_x)$ such that
 $$u\cdot((\nabla^{g})^iR^g)_x=0 \ ,\ (u,v)\cdot((\nabla^{g}\otimes \nabla^{A_0})^iF^{A_0})_x=0 \  \, \hbox{for} \ \, 0\leq i\leq k\,.
 $$
Note that $\hg^{g,A_0}_x{(k)}$ is a Lie subalgebra of $\so(T_xM)\oplus\ad({P}_x)$, and that for any $k\in\N$  
$\hg^{g,A_0}_x{(k+1)}\subset \hg^{g,A_0}_x{(k)}\,.$
Put 
$$k_x^{g,A_0}\coloneqq \min\{k\in\N|\ \hg^{g,A_0}_{x}{(k+1)}= \hg^{g,A_0}_x{(k)}\}\,.$$
Using these definitions and notations we see that
\begin{re}  \label{IHTDef2}
	
	The triple $(g,P\textmap{p} M,A_0)$  is   infinitesimally homogeneous if  and only if for any $(x_1,x_2)\in M\times M$, there exists a pair $(f,\phi)$ where $f:T_{x_1}M\to T_{x_2}M$ is a linear isometry, and $\phi:P_{x_1}\to P_{x_2}$ is a $K$-equivariant isomorphism,  such that for any $0\leq k\leq k_{x_1}^{g,A_0}+1$ one has:
	
	\begin{enumerate}
		\item $f\big(\big\{(\nabla^{g})^k R^g\big\}_{x_1}\big)=\big\{(\nabla^{g})^kR^g\big\}_{x_2}$.
		\item $(f,\phi)\big(\big\{ (\nabla^{g}\otimes \nabla^{A_0})^kF^{A_0}\big\}_{x_1}\big)=\big\{(\nabla^{g}\otimes \nabla^{A_0})^kF^{A_0}\big\}_{x_2}$.
	\end{enumerate}

\end{re}
The  second condition can be formulated explicitly as follows:
   $(\nabla^{g}\otimes \nabla^{A_0})^kF^{A_0}$ is a section of the vector bundle $(\Lambda_M^1)^{\otimes k}\otimes\Lambda_M^2\otimes\ad(P)$. The fiber    $\ad(P)_x$ at $x\in M$ for the adjoint bundle $\ad(P)$ is given by $\ad(P)_x=(P_x\times \kg) /K\,,$
so a $K$-equivariant isomorphism  $\phi: P_{x_1}\to P_{x_2}$  induces a linear isomorphism 
$$\ad(\phi):\ad(P)_{x_1}\to \ad(P)_{x_2} \ ,\ [y,\alpha]\mapsto [\phi(y),\alpha].$$
 The second condition in Remark \ref{IHTDef2}  can be written as
\begin {equation}
\begin{split}
\big\{(\nabla^{g}\otimes \nabla^{A_0})^k F^{A_0} \big\}_{x_2}\big(f(v_1),\dots,f(v_k),f(w_1),f(w_2)\big) =\\
=\ad(\phi)\big( \{ (\nabla^{g}\otimes \nabla^{A_0})^k F^{A_0} \}_{x_1}(v_1,\dots,v_k,w_1,w_2)\big) 
\end{split}
\end {equation}
for any tangent vectors  $v_i\in T_{x_1}M$, $w_1$, $w_2\in T_{x_1}M$\,. \\

The    isomorphism pair $(f,\phi)$  defines a Lie algebra isomorphism 
$$ \so(T_{x_1}M)\oplus\ad({P}_{x_1})\to\so(T_{x_2}M)\oplus\ad({P}_{x_2})\,, $$
which isomorphically maps $\hg_{x_1}^{g,A_0}{(k)}$ onto $\hg^{g,A_0}_{x_2}{(k)}$ for $0\leq k\leq k_{x_1}^{g,A_0}+1$. This implies
\begin{re}
	Let $(g,P\textmap{p} M,A_0)$  be an infinitesimally homogeneous triples. Then $k_x^{g,A_0}$ is independent of $x$. We will denote by $k^{g,A_0}$ the obtained constant.	
\end{re}

Applying Theorem \ref{ConnectionB-Th} to our situation we obtain:

\begin{thry} \label{nablaCA-FA-nabla-alpha}
	Suppose   $(g,P\textmap{p} M,A_0)$  is  an infinitesimally homogeneous triple on $(M,g)$. Then, there exists a pair $(\nabla,A)$ consisting of a metric connection on $M$ and $A\in {\cal A}(P)$ with the following properties:
	$$\nabla  R^\nabla=0,\ \nabla T^\nabla=0,\ (\nabla\otimes \nabla^A)F^A=0,\  (\nabla\otimes \nabla^A)(A-A_0)=0\,.
	$$
\end{thry}
\begin{proof}
Let $Q\coloneqq \O(M)\times_M P$ and $E\coloneqq Q\times_{G} V$ be as above. Let $\sigma\in \Gamma(E)$ be the section defined by the pair $(R^g,F^{A_0})$. By Corollary \ref{ConnectionBCo} there exists a connection $B=(C,A)\in \cal A(\O(M))\times \cal A(P)={\cal A}(Q)$	such that: 
	$$\nabla^{B}_\EE \sigma=0\ , \   (\nabla^{B}_{\MM}\otimes\nabla^B_\add)(B-B_0)=0\,.
$$ 
Let $\nabla$ denote the induced metric connection on $M$ by $C$ and $\nabla^A$ denote the induced connection on the adjoint bundle $\ad(P)$  by $A$. From
$\nabla^{B}_\EE \sigma =0$ we obtain
\begin{equation}
\nabla R^g =0,\ (\nabla\otimes \nabla^A )F^{A_0}=0\,.
\end{equation}
The difference $B-B_0$ can be identified with the pair $(C-C_0,A-A_0)$. Hence, equation $(\nabla^{B}_{\MM}\otimes\nabla^B_\add)(B-B_0)=0$ is equivalent to
\begin{equation}
\nabla(C-C_0)=0\,,\ (\nabla\otimes \nabla^A )(A-A_0)=0\,,
\end{equation}
The conditions $\nabla R^g=0\,$, $\nabla (C-C_0)=0$ and Corollary \ref {C_C_0CO} imply $\nabla R^\nabla=0\,$, $\nabla T^\nabla=0$. The conditions $\nabla T^\nabla=0$ ,   $(\nabla\otimes \nabla^A )(A-A_0)=0$, $(\nabla\otimes \nabla^A )F^{A_0}=0$ and Lemma \ref{NablaFNablaF0} below give $(\nabla\otimes \nabla^A )F^{A}=0$.
\end{proof}
\begin{lm}\label{NablaFNablaF0}
	Retain the notation used above. If $\nabla T^\nabla=0$, $\,(\nabla \otimes \nabla^A )(A-A_0)=0$, then $(\nabla \otimes \nabla^A )F^{A_0}=0$ if and only if $(\nabla \otimes \nabla^A )F^{A}=0$.
\end{lm}
\begin{proof}
Put $\alpha\coloneqq A-A_0\in A^1(M,\ad(P))$. One has the following variation formula for the curvature 2-form
	\begin{equation}\label{FAplusAlpha}
	F^{A_0}=F^{A-\alpha}=F^ {A}-d^A\alpha+\frac{1}{2}[\alpha\wedge\alpha]\,.
	\end{equation}
	Let $X$, $Y$, $Z\in {\mathfrak X}(M)$ be arbitrary vector fields on $M$. In one hand,
	\begin{equation}\label{dAalpha11}
		(d^A\alpha)(X,Y) =\nabla^A_X(\alpha(Y))-\nabla^A_Y(\alpha(X))-\alpha([X,Y])\,.
	\end{equation}
On the other hand, since  $\alpha\in A^1(M,\ad (P))$ is $(\nabla \otimes \nabla^A )$-parallel, one has
	\begin{equation}\label{nablaCAalpha1}
		0=((\nabla \otimes \nabla^A )\alpha)(X,Y)=\nabla_X^A(\alpha(Y))-\alpha(\nabla_XY)\,.
	\end{equation}
	Using (\ref{dAalpha11}) and (\ref{nablaCAalpha1}), one obtains
	\begin{equation}\label{dAalpha12} 
		(d^A\alpha)(X,Y)=\alpha(\nabla_XY)-\alpha(\nabla_YX)-\alpha([X,Y])=\alpha(T^\nabla (X,Y))\,.
	\end{equation}
	and from (\ref{nablaCAalpha1}) one gets
	 $\nabla_X^A(\alpha(T^\nabla (Y,Z)))=\alpha(\nabla_XT^\nabla (Y,Z))\,.$
Using  (\ref{nablaCAalpha1}) and (\ref{dAalpha12}) one obtains 
	\begin{equation*}
		\begin{split}
			((\nabla \otimes \nabla^A ) d^A\alpha)(X,Y,&Z)=\nabla^A_X ((d^A\alpha)(Y,Z))- (d^A\alpha)(\nabla_XY,Z)-  (d^A\alpha)(Y,\nabla_XZ)\\
			&=\nabla^A_X (\alpha(T^\nabla (Y,Z)))- \alpha(T^\nabla (\nabla_XY,Z))-\alpha(T^\nabla (Y,\nabla_XZ))\\
			&= \alpha(\nabla_XT^\nabla(Y,Z))- \alpha(T^\nabla(\nabla_XY,Z))-\alpha(T^\nabla(Y,\nabla_XZ))\\
			&=\alpha((\nabla_X T^\nabla)(Y,Z))=0.
		\end{split}
	\end{equation*}
	Put $\eta=\frac{1}{2}[\alpha\wedge\alpha]$. For arbitrary vector fields $Y,Z\in \mathfrak X(M)$
	\begin{equation}\label{etacroshe}
		2\eta(Y,Z)=[\alpha\wedge\alpha](Y,Z)=[\alpha(Y),\alpha(Z)]-[\alpha(Z),\alpha(Y)]=2[\alpha(Y),\alpha(Z)]\,.
	\end{equation}
	Therefore,
	\begin{equation}\label{nablaetacroshe}
		\nabla^{A}_X(\eta(Y,Z))=\nabla^{A}_X([\alpha(Y),\alpha(Z)])=[\nabla^{A}_X(\alpha(Y)),\alpha(Z)]+[\alpha(Y),\nabla^{A}_X(\alpha(Z))]\,.
	\end{equation}
	Using (\ref{nablaCAalpha1}) and  (\ref{nablaetacroshe}) we obtain 
	\begin{equation*}
		\begin{split}
			((\nabla \otimes & \nabla^A )_X\eta)(Y,Z)= \nabla^{A}_X(\eta(Y,Z))-\eta(\nabla_XY,Z)-\eta(Y,\nabla_XZ)\\
			=&[\nabla^{A}_X(\alpha(Y)),\alpha(Z)]+[\alpha(Y),\nabla^{A}_X(\alpha(Z))] -[\alpha(\nabla_XY),\alpha(Z)]
			 -[\alpha(Y),\alpha(\nabla_XZ)]\\
			 =&[((\nabla \otimes  \nabla^A )_X \alpha)(Y),\alpha(Z)]+[\alpha(Y), ((\nabla \otimes  \nabla^A )_X \alpha)(Z)]=0\,.
		\end{split}
	\end{equation*}
So, we proved that the $\ad$-valued 2-forms $d^A\alpha$ and  $[\alpha\wedge\alpha]$ are $(\nabla \otimes \nabla^A )$-parallel. Using (\ref{FAplusAlpha}), we conclude: $(\nabla \otimes \nabla^A )F^{A_0}=0$ if and only if $(\nabla \otimes \nabla^A )F^{A}=0$.
\end{proof}

Using Ambrose-Singer's theorem (see Theorem \ref{ASTheorem}) we obtain:
\begin{re}
Let $(g,P\textmap{p} M,A_0)$  be an infinitesimally homogeneous triple over $(M,g)$. Then, the Riemannian metric $g$ is locally homogeneous.
\end{re}
A pair $(\nabla,A)$ of a metric connection on $M$ and $A\in {\cal A}(P)$ defines in a canonical way a linear connection $\bar\nabla$ on the tangent bundle $T_P$. This construction will play a major role in what follows, so we explain this construction in detail: 

Recall that, in general, a connection $A$ on the principal $K$-bundle $p:P\to M$ can be identified with the corresponding horizontal distribution. Using the  bundle isomorphism $J: p^*T_M\to A$
 (defined by  the inverse of $\resto{p_*}{A}$), we obtain a linear connection $\nabla^{h,A}\coloneqq J(p^*\nabla)$ on the horizontal subbundle $A\subset T_P$. 
%
%

On the other hand we can define a linear connection $\nabla^{v,A}$ on the vertical bundle $V_P$ via the canonical bundle isomorphism $V_P\simeq P\times \kg$. More precisely, if $a^\#$, $b^\#$ denote the fundamental fields corresponding to $a,b\in \kg$ and if $\tilde X$ denotes the $A$-horizontal lift of $X\in\mathfrak X(M)$, then we define
$$\nabla^{v,A}_{a^\#}b^\#=[a,b]^\# , \ \nabla^{v,A}_{\tilde X}a^\#=0\,.
$$

	Using the direct sum decomposition $T_P=A\oplus V_P$, the linear connection $\nabla^{h,A}$ on the horizontal subbundle $A$, and the linear connection $\nabla^{v,A}$ on the vertical subbundle $V_P$, we obtain a linear connection on $P$ defined by 
	$$\bar\nabla\coloneqq \nabla^{h,A}\oplus \nabla^{v,A}.
	$$ 

The straightforward calculations show that \cite [Proposition 2.6]{Ba2}, the connection $\bar \nabla$ has the following properties:

\begin{pr}\label{nablabarPr}
	Let $\bar \nabla$ be the linear connection on the tangent bundle $T_P$  defined as above. Let $\tilde X$, $\tilde Y$ be the $A$-horizontal lift of vector fields $X$, $Y \in \mathfrak X(M)$ and  $Z$ be a vertical vector field on $P$. Then 
	\begin{enumerate}
		\item $\bar \nabla_{\tilde X}\tilde Y=(\nabla_XY)^{\sim},\ \bar \nabla_{\tilde X}Z=[\tilde X, Z].$
		\item $\bar \nabla_{a^\#}b^\#=[a,b]^\# \, ,  \bar \nabla_{a^\#}Z=[a^\#, Z].$
		\item $ \bar \nabla_{\tilde X}a^\#=\bar \nabla_{a^\#}\tilde X=[a^\#,\tilde X]=0.$

	\end{enumerate}
\end{pr}
%

The canonical bundle isomorphism $\xi: P\times \kg\to V_P$ can be used to identify the space of smooth maps ${\cal C}^\infty(P, \kg)$ with the Lie algebra of vertical vector fields on $P$. This identification is given by $f\mapsto \xi(f)$, where, for $f\in {\cal C}^\infty(P, \kg)$, $\xi(f):P\to V_P$ is  the vertical vector field given by
$$\xi(f)_y\coloneqq f(y)^\#_y=\at{\frac{d}{dt}}{t=0}\big(y\exp t f(y)\big).
$$
%

%

The following proposition can be proved easily \cite [Lemma 2.7, Prop. 2.8.] {Ba2}.
\begin{pr}\label{XTILDEXINU}
	Let $p:P\to M$ be a principal $K$-bundle and $A\in {\cal A}(P)$. Let $\nu\in A^0(\ad (P))$ and $a^\#$ denote the fundamental vector field corresponding to $a\in\kg$. Let $\tilde X$ be the $A$-horizontal lift of vector field $X\in \mathfrak X(M)$. Then, one has
	$$\bar\nabla_{a^\#}\xi(\nu)=[a^\#,\xi(\nu)]=0\,,\ \bar \nabla_{\tilde X}\xi(\nu)=[\tilde{X},\xi(\nu)]=\xi(\nabla^A_X\nu).$$
\end{pr}
%
%
%

%
\begin{thry} \label{nablabarIsSinger}
	Let  $(\nabla,A)$ be a pair consisting of a metric connection on $M$ and a connection $A\in {\cal A}(P)$. Suppose $\nabla  R^\nabla =0$, $\nabla T^\nabla=0$, $(\nabla\otimes \nabla^A)F^A=0$. Then, the associated connection $\bar\nabla\coloneqq \nabla^{h,A}\oplus \nabla^{v,A}$ satisfies the following conditions:
	$$\bar \nabla R^{\bar \nabla}=0\,, \, \bar \nabla T^{\bar \nabla}=0.$$
\end{thry}
\begin{proof} 
%
%
 Denote the torsion tensor and the curvature tensor of connection $\bar\nabla$ by $\bar T$, $\bar R$ to save on notation. Let $a^\#$, $b^\#$ be the fundamental vector fields corresponding to $a,b\in \kg$, and let $\tilde X$, $\tilde Y$ denote the $A$-horizontal lifts of vector fields  $X$, $Y$ on $M$. 
The properties of $\bar\nabla$ in Proposition \ref{nablabarPr} imply
\begin{equation}\label{TorsionNablaBar1}
\bar T(\tilde X,a^\#)=0\, , \ \bar T(a^\#,b^\#)=[a^\#,b^\#]\,.
\end{equation}

Using the equation $\xi(F^A(X,Y))= [X,Y]^{\sim}-[\tilde X,\tilde Y]$ for the vertical component of the Lie bracket of two horizontal lifts \cite[P. 257]{GH}, one obtains
\begin{equation}\label{TorsionNablaBar2}
 \bar T(\tilde X,\tilde Y)=T^{\nabla}(X,Y)^{\sim}+\xi(F^{A}(X,Y)). 
\end{equation}
%
		
		%
		%
		%
		%
		%
		%
		%
		%
		
To prove $\bar\nabla \bar T=0$, it suffices to show that $(\bar\nabla\bar T)(U,V,W)=0$ for vector fields $U$, $V$, $W\in\mathfrak X(P)$ in the special cases when each one of these three vector fields is either a  horizontal lift, or a fundamental vector field. 
One has
	\begin{equation}\label{barUTVW}
	(\bar \nabla_U\bar T)(V,W)=\bar \nabla_U\bar T(V,W)-\bar T(\bar \nabla_UV,W)-\bar T(V,\bar \nabla_UW)\,.
	\end{equation}
If $U=\tilde X$, $V=\tilde Y$, $W=\tilde Z$ then 
		\begin{equation}\label{nablaUTVW}
			\begin{split}
				\bar \nabla_{U}\bar T(V,W)&=\bar \nabla_{\tilde X}\big((T^{\nabla}(Y,Z)^{\sim})+\xi (F^A(Y, Z))\big)\\
				&=\big(\nabla_{X}T^{\nabla}(Y,Z)\big)^{\sim}+ \xi( \nabla^A_XF^A(Y, Z)).
			\end{split}
		\end{equation}
		In the same way one obtains 
			\begin{equation}\label{barnablaUVW}
			\bar T(\bar \nabla_UV,W)=(T^\nabla(\nabla_{X}Y,Z))^{\sim}+\xi(F^A(\nabla_X Y, Z))\,,
	        \end{equation}
		\begin{equation}\label{barTnablaVUW}
		\bar T(V,\bar \nabla_UW)=(T^\nabla(Y,\nabla_{X}Z))^{\sim}+\xi (F^ A( Y, \nabla_XZ))\,.
	     \end{equation}
		Using the equations (\ref{barUTVW}), (\ref{nablaUTVW}), (\ref{barnablaUVW}), (\ref{barTnablaVUW})  and Proposition \ref{XTILDEXINU} :
		\begin{equation*} 
			\begin{split}
				(\bar \nabla_U\bar T)(V,W) & =\bar \nabla_U\bar T(V,W)-\bar T(\bar \nabla_UV,W)-\bar T(V,\bar \nabla_UW) \\
				& =(\nabla_{X}T^{\nabla}(Y,Z))^{\sim}-(T^{\nabla}(\nabla_XY,Z))^{\sim} -(T^{\nabla}(Y,\nabla_XZ))^{\sim}\\
				&\ + \xi(\nabla^A_XF^A(Y, Z))-\xi (F^A(\nabla_XY, Z))-\xi (F^A( Y, \nabla_XZ))\\
				&=\big((\nabla T^{\nabla})(X,Y,Z)\big)^{\sim}+\xi\big((\nabla\otimes \nabla^A)F^ A)(X,Y,Z)\big)\,.
			\end{split}
		\end{equation*}  
		As $\nabla T^\nabla=0$ and $(\nabla\otimes \nabla^A)F^ A=0$, the right hand side of above equation vanishes. 
		
		If  $U=a^\#$, $V=b^\#$, $W=c^\#$, where $a$, $b$, $c\in\kg$ then by the Jacobi identity 
		$$(\bar \nabla_{a^\#}\bar T)(b^\#,c^\#)=\big([a,[b,c]]+[b,[c,a]]]+[c,[a,b]]\big)^\#=0.
		$$
		%
		%
		%
		%
If $U=a^\#$ and $V=\tilde Y$, $W=\tilde Z$ are the $A$-horizontal lifts of vector fields $Y, Z\in \mathfrak X(M)$ then using Propositions \ref{XTILDEXINU} , \ref{nablabarPr} we get
		$$\bar \nabla_{a^\#}\big((T^{\nabla}(Y,Z))^{\sim}+\xi (F^A(Y,Z))\big)-\bar T(\bar \nabla_{a^\#}\tilde Y,\tilde Z)-\bar T(\tilde Y,\bar \nabla_{a^\#}\tilde Z)=0.$$
In any other possible cases the claim follows directly from the definition of $\bar\nabla$. For the curvature tensor the only non vanishing cases are:
\begin{equation} \label{CurvatureNablaBar}
 \bar R(\tilde X,\tilde Y)\tilde Z= (R^\nabla( X, Y) Z)^\sim\, ,  \ \bar R(\tilde X,\tilde Y)c^\#=\bar \nabla _{\xi(F^A(X,Y))}c^\#. 
\end{equation}
and with the same kind of arguments as above, one obtains $\bar \nabla \bar R=0$.
\end{proof}
Recall that the space of connections ${\cal A}(P)$ for a principal $K$-bundle $P$ over $M$ is an affine space with model space $A^1(\ad(P))$ \cite{DK}. For a connection $A_0\in {\cal A}(P)$, and a 1-form $\alpha\in A^1(\ad(P))$ the connection form $\omega_A$ of $A\coloneqq A_0+\alpha$ is given by $\omega_A=\omega_{A_0}+\alpha$, where the second term on the right has been identified with the associated tensorial 1-form of type $\ad$ on $P$ (see \cite[Example 5.2 p.76]{KN}). In other words, for a tangent vector $v\in T_yP$ the element $\alpha_y(v)\in\kg$ is defined by the equality   $\alpha(p_*(v))=[y,\alpha_y(v)]\in \ad(P_{p(y)})$. So, regarding $\alpha(p_*(v))$ as a $K$-equivariant map $P_{p(y)}\to \kg$, one has $\alpha_y(v)=\alpha(p_*(v))(y)$.
 \begin{lm}\label{Aliftxi}
	Let $A=A_0+\alpha$, where $A,A_0\in \cal A(P)$ and $\alpha\in A^1(\ad(P))$. Let $Z$ be a vector field on $M$, $\tilde Z^{A_0}$ be its $A_0$-horizontal lift, and $\tilde Z^A$ be its $A$-horizontal lift (which coincides with the $A$-horizontal projection of $\tilde Z^{A_0}$).  Then   
	$$\tilde Z^A=\tilde Z^{A_0}-\xi(\alpha(Z))\,.
	$$
\end{lm}
\begin{proof}
	For any vector field $Y\in{\mathfrak X}(P)$  the $A$-horizontal projection of $Y$ is given by $Y^A=Y-\xi(\omega_A(Y))$. 
	If now  $Y=\tilde Z^{A_0}$,  we have $\omega_{A_0}(Y)=0$, so for any $y\in P$ we have $\omega_{A,y}(Y)=\alpha_y(Y)=\alpha(Z)(y)$.
\end{proof}
\begin{lm}\label{AA0parallel} Let  $\nabla$ be a linear connection on $M$ and $A\in {\cal A}(P)$ such that
$$\nabla R^\nabla =0\, , \nabla T^\nabla =0\ ,\ (\nabla\otimes \nabla^A) F^A=0\,, (\nabla\otimes \nabla^A) (A-A_0)=0\,.
$$
Then the distributions $A$, $A_0$ are $\bar\nabla$-parallel, where $\bar\nabla=\nabla^{h,A}\oplus \nabla^{v,A}$.
\end{lm}
\begin{proof}
To prove that $A$ is $\bar\nabla$-parallel, let $\tilde X$ be the $A$-horizontal lift of vector field $X$ on $M$. We have to show that for any vector field $Y$ on $P$ the vector field $\bar\nabla_Y\tilde X$ is $A$-horizontal. If $Y\coloneqq \tilde Z$ is $A$-horizontal lift of a vector fields $Z$ on $M$ then $\bar\nabla_{Y}\tilde X$ is $A$-horizontal lift of $\nabla_ZX$. If $Y\in\mathfrak X(P)$ is a vertical vector field, then $\bar\nabla_{Y}\tilde X=0$ and therefore $\bar\nabla_{Y}\tilde X\in\Gamma(A)$.
		
For the second property, we will denote by  $\tilde Z^{A_0}$ the  $A_0$-horizontal lift of a vector field   $Z\in{\mathfrak X}(M)$. Since the vector fields of the form  $\tilde Z^{A_0}$ generate $\Gamma(A_0)$ as a ${\cal C}^\infty(P,\R)$-module, it suffices to prove that, for arbitrary vector fields $Y\in\mathfrak X(P)$ and $Z\in{\mathfrak X}(M)$   one has $\bar\nabla_Y(\tilde Z^{A_0})\in\Gamma(A_0)$. Putting $\alpha=A-A_0$ and using Lemma \ref{Aliftxi} we obtain
		$$\tilde Z^A= \tilde Z^{A_0}-\xi(\alpha(Z)).$$
If $Y=\tilde U$ is the  $A$-horizontal lift of a vector field $U$ on $M$, then, using the parallelism assumption $(\nabla\otimes \nabla^A)\alpha=0$, the definition of $\bar\nabla$ and  Proposition \ref{XTILDEXINU},  we obtain
		\begin{equation*}
			\begin{split}
				\bar\nabla_{Y}\tilde Z^{A_0}&=\bar\nabla_{Y}\big(\tilde Z^A+\xi(\alpha(Z))\big)=(\nabla_UZ)^{\sim A}+\bar\nabla_{\tilde U}\xi(\alpha(Z))\\
				&=(\nabla_UZ)^{\sim A}+ \xi\big(\nabla^A_U(\alpha(Z))\big)=(\nabla_UZ)^{\sim A}+\xi(\alpha(\nabla_UZ))\,.
			\end{split}
		\end{equation*}
		But, by   Lemma \ref{Aliftxi}, the right hand side of the above equation is the $A_0$-horizontal lift of $\nabla_UZ$, which proves the claim in this case.  If now $Y=a^\#$ is a fundamental field, then 
		\begin{equation*}
				\bar\nabla_{Y}\tilde Z^{A_0}=\bar\nabla_{a^\#}\big(\tilde Z^A+\xi(\alpha(Z))\big)=\bar\nabla_{a^\#}\tilde Z^A+\bar\nabla_{a^\#}\xi(\alpha(Z))\,.				
		\end{equation*}
		By Proposition \ref{nablabarPr}, $\bar\nabla_{a^\#}\tilde Z^A=0$. Also, as $\alpha(Z)\in A^0(\ad (P))$, using Proposition \ref{XTILDEXINU} we obtain $\bar\nabla_{a^\#}\xi(\alpha(Z))=0$. Therefore, $\bar\nabla_{Y}\tilde Z^{A_0}$ is $A_0$-horizontal. 
\end{proof}
In conclusion, using Theorems \ref{nablaCA-FA-nabla-alpha}, Theorem \ref{nablabarIsSinger}, Lemma \ref{AA0parallel} and Proposition \ref{nablabarPr} we obtain

\begin{thry} \label{ConnectionsCA} Let $(g,P\textmap{p} M,A_0)$  be  infinitesimally homogeneous. There exists a pair $(\nabla,A)$ consisting of a metric connection on $(M,g)$ and $A\in {\cal A}(P)$ such that
$$\nabla  R^\nabla=0,\ \nabla T^\nabla=0,\ (\nabla\otimes \nabla^A)F^A=0,\  (\nabla\otimes \nabla^A)(A-A_0)=0\,.
$$
The  linear connection $\bar\nabla=\nabla^{h,A}\oplus \nabla^{v,A}$ on the tangent bundle $T_P$ associated with this pair has the following properties:
\begin{enumerate}
\item $\bar \nabla R^{\bar \nabla}=\bar \nabla T^{\bar \nabla}=0$,
\item The vector fields $a^\#$ are $\bar\nabla$ parallel along the $A$-horizontal curves,
\item The distributions $A$, $A_0$ are $\bar\nabla$-parallel.
\end{enumerate}
\end{thry}

\subsection{Structure and classification theorems}

We begin by recalling  the  results proved in \cite[Ch.VI Section 7]{KN} on the existence and extension properties of  local affine isomorphisms with respect to a linear connection  satisfying the conditions $\nabla R^\nabla=\nabla T^\nabla=0$. Compared to  \cite{KN},  our presentation uses a new formalism:  the space of germs of $\nabla$-affine isomorphisms.

Let $M$ be a  differentiable $n$-manifold, and let $\nabla$ be a linear connection on $M$ satisfying the conditions $\nabla R^\nabla=\nabla T^\nabla=0$. Let $\Sg^\nabla$ be the space of germs of $\nabla$-affine isomorphisms defined between open sets of $M$, and let $\sg:\Sg^\nabla\to M$, $\tg:\Sg^\nabla\to M$ be the source, respectively the target map on this space.   $\Sg^\nabla$ has a canonical structure of a differentiable manifold  and, with respect to this structure, $\sg$ and $\tg$ are local diffeomorphisms.  A germ $\varphi\in \Sg^\nabla$ defines an isomorphism $\varphi_*:T_{\sg(\varphi)}M\to T_{\tg(\varphi)}M$ with the property 
$$\varphi_*(R^\nabla_{\sg(\varphi)})=R^\nabla_{\tg(\varphi)},\ \varphi_*(T^\nabla_{\sg(\varphi)})=T^\nabla_{\tg(\varphi)}\,.
$$ 
Conversely,  \cite[Theorem 7.4, p. 261]{KN} can be reformulated as follows
\begin{re}\label{fphi}
Let $(u,v)\in M\times M$. For any linear isomorphism $f:T_uM\to T_vM$ satisfying
\begin{equation}\label{TRcompat}
f(R^\nabla_{u})=R^\nabla_{v},\ f(T^\nabla_{u})=T^\nabla_{v}	
\end{equation}
there exists a unique germ $\varphi_f\in (\sg,\tg)^{-1}(u,v)$ (of a $\nabla$-affine isomorphism) such that $(\varphi_f)_*=f$.  
\end{re}

Moreover, taking into account \cite[Corollary 7.5, p. 262]{KN} we see that
\begin{re}\label{Trans}
Suppose that $M$ is connected. The map $(\sg,\tg):\Sg^\nabla\to M\times M$ is surjective. 
\end{re}
\begin{proof}
Indeed, let $(u,v)\in M\times M$. Parallel transport along a path $\gamma:[0,1]\to M$ joining $u$ to $v$ defines a linear isomorphism $f:T_uM\to T_vM$ satisfying (\ref{TRcompat}). The claim follows from Remark \ref{fphi}.
	
\end{proof}

Let now $\sigma:M\times M\to M$, $\tau:M\times M\to M$ be the  projections on the two factors, and let $\mathrm{Iso}(\sigma^*(T_M),\tau^*(T_M))\subset \Hom(\sigma^*(T_M),\tau^*(T_M))$ be the (locally trivial) fibre bundle of isomorphisms between the two pull-backs of $T_M$. Conditions (\ref{TRcompat}) define  a closed, locally trivial subbundle $S^\nabla\subset \mathrm{Iso}(\sigma^*(T_M),\tau^*(T_M))$. Remark  \ref{fphi} shows that %
\begin{re}
The natural map $\delta: \Sg^\nabla\to S^\nabla$ given by $\varphi\mapsto\varphi_*$ is bijective.	
\end{re}
This map is an injective immersion,  and it is bijective, but it is {\it not} a diffeomorphism. $\Sg^\nabla$ can be identified with the union of leaves of a foliation of $S^\nabla$ with $n$-dimensional leaves.  The topology of $\Sg^\nabla$ is finer than the topology of $S^\nabla$. The leaves of this foliation are the integrable submanifolds of the involutive distribution ${\cal D}^\nabla\subset T_{S^\nabla}$ defined in the following way: Let $f\in S^\nabla$. Put $u\coloneqq \sigma(f)$, $v\coloneqq \tau(f)$,   $\varphi\coloneqq \delta^{-1}(f)$. For a tangent vector $\xi\in  T_uM$, let $\gamma:(-\varepsilon,\varepsilon)\to M$ be a smooth curve such that $\gamma(0)=u$, $\dot\gamma(0)=\xi$. Using parallel transport with respect to $\nabla$ along the curves $\gamma$, $\varphi\circ\gamma$ we obtain, for any sufficiently small $t\in (-\varepsilon,\varepsilon)$, isomorphisms $a_t:T_uM\to T_{\gamma(t)}M$, $b_t:T_vM\to T_{\varphi(\gamma(t))}M$. Define
$$\lambda_f(\xi)\coloneqq \at{\frac{d}{dt}}{t=0} (b_t\circ f\circ a_t^{-1})\in T_{(u,v)}(S^\nabla)\,.
$$ 
The distribution ${\cal D}^\nabla$ is defined by 
$${\cal D}_f^\nabla\coloneqq \{\lambda_f(\xi)|\ \xi\in T_{\sigma(f)}M\}\,.$$
The curve $t\mapsto b_t\circ f\circ a_t^{-1}$ will be an integral curve of this distribution.  
Note that we may take $\gamma$ to be the $\nabla$-geodesic with initial condition $(u,\xi)$, and then $\varphi\circ\gamma$ will be  the $\nabla$-geodesic with initial condition $(v,f(\xi))$. Using this remark, we obtain
\begin{re}
Let $\varphi\in \Sg^\nabla$, and $\xi\in T_{\sg(\varphi)}M$. Suppose that $\nabla$-geodesics $\gamma$, $\eta$ with initial conditions $(\sg(\varphi),\xi)$, $(\tg(\varphi),\varphi_*(\xi))$  	respectively can be both extended on the interval $(\alpha,\beta)\ni 0$.  Then $\gamma$ has a  smooth lift in $\Sg^\nabla$ with initial condition $\varphi$ via the source map $\sg:\Sg^\nabla\to M$.
\end{re}
Using this remark one can prove:
\begin{pr}\label{globalAff}
 Let $\nabla$ be a connection on a connected manifold $M$ such that $\nabla R^\nabla=\nabla T^\nabla=0$. Suppose  that $\nabla$ is complete. Then the source map $\sg:\Sg^\nabla\to M$ is a covering map. In particular, when $\nabla$ is complete and $M$ is simply connected, any element $\varphi\in \Sg^\nabla$ extends to a unique global $\nabla$-affine  isomorphism $M\to M$. In particular, for any pair $(x,x')\in M\times M$ there exists a unique global $\nabla$-affine  isomorphism mapping $x$ to $x'$.
\end{pr}
\begin{proof}
It suffices to note, that for a $\nabla$-convex open set $U\subset M$ the following holds: any   germ $\varphi\in\Sg^\nabla$ with $\sg(\varphi)\in U$ has an extension on $U$. The point is that, since $\nabla$ is complete, for any geodesic $\gamma:(\alpha,\beta)\to U$ passing through $\sg(\varphi)$, the composition $\varphi\circ \gamma$ can be extended on the whole $(\alpha,\beta)$.   Therefore all  $\nabla$- geodesics in $U$ passing through $\sg(\varphi)$ admit lifts  with initial condition $\varphi$.  Using these lifts it follows that the connected components of $\sg^{-1}(U)$ are identified with $U$ via $\sg$.
\end{proof}

An important special case (which intervenes in the proof of Singer's theorem) concerns a linear connection $\nabla$ satisfying the conditions $\nabla R^\nabla=\nabla T^\nabla=0$ which is a metric connection, i.e. there exists a Riemannian metric $g$ on $M$ such that $\nabla g=0$. In this case one defines submanifolds 
$$\Sg^\nabla_g\coloneqq \{\varphi\in \Sg^\nabla|\ \varphi_* \hbox{ is an isometry}\},\ S^\nabla_g\coloneqq \{f\in S^\nabla|\ f \hbox{ is an isometry}\}$$
of $\Sg^\nabla$, $S^\nabla$ respectively. $\Sg^\nabla_g$ is open in  $\Sg^\nabla$. In other words  $S^\nabla_g$ is a union of integral submanifolds (of maximal dimension) of the involutive distribution ${\cal D}^\nabla$. On the other hand, an important result in Riemannian geometry states \cite[Proposition 1.5]{TV}:
\begin{pr}\label{completeNabla} Let $(M,g)$ be a complete Riemannian manifold. Then any metric connection $\nabla$ on $M$ is complete.
\end{pr} 
Using these facts Proposition \ref{globalAff} combined with the proof of Remark \ref{Trans} gives:

\begin{co} \label{globalAffg} Let $(M,g)$ be a connected  Riemannian manifold  endowed with a metric connection $\nabla$ such that  $\nabla R^\nabla=\nabla T^\nabla=0$. Suppose that $(M,g)$ is complete. Then the source map $\sg:{\Sg}^\nabla_g\to M$ is a covering map. In particular, when $(M,g)$ is complete and   simply connected, any element $\varphi\in \Sg^\nabla_g$ extends to a unique global $\nabla$-affine  isometry $M\to M$. In particular, for any pair $(x,x')\in M\times M$ there exists a unique global $\nabla$-affine  isometry mapping $x$ to $x'$.	
\end{co}

Now we come back to the connection $\bar\nabla=\nabla^{h,A}\oplus \nabla^{v,A}$ on $T_P$ associated with a pair $(\nabla,A)$ of a metric connection on $M$ and $A\in{ \cal A} (P)$ which satisfy the conditions
$$\nabla  R^\nabla=0,\ \nabla T^\nabla=0,\ (\nabla\otimes \nabla^A)F^A=0,\  (\nabla\otimes \nabla^A)(A-A_0)=0\,.
$$
We know that $\bar \nabla R^{\bar\nabla}=\bar \nabla T^{\bar\nabla}=0$, so all constructions and results above apply to $\bar\nabla$.  Using the additional structure we have on $P$ (the $K$-action, the two connections $A$, $A_0$,  and the metric $g$ on $P/K$) we will define  an  open submanifold  $\Sg_{g,K}^{\bar\nabla}$ of $\Sg^{\bar\nabla}$ consisting of germs of affine transformations which are compatible with this structure:

Since $\bar\nabla$ is $K$-invariant, the manifolds $\Sg^{\bar\nabla}$, $S^{\bar\nabla}$ come with natural right $K$-actions given by $(\varphi,k)\mapsto R_k\circ\varphi\circ R_k^{-1}$, $(f,k)\mapsto (R_k)_*\circ f \circ (R_k^{-1})_*$, and the  maps $\sg$, $\tg:\Sg^{\bar\nabla}\to P$, $\sigma$, $\tau:S^{\bar\nabla}\to P$ are $K$-equivariant.

We define a submanifold $S^{\bar\nabla}_{K,g}\subset S^{\bar\nabla}$ by
\begin{equation}
\begin{split}
&S^{\bar\nabla}_{g,K}\coloneqq \big\{f\in S^{\bar\nabla}|\ f(A_{\sigma(f)})=A_{\tau(f)}\ ,\, f(A_{0,\sigma(f)})=A_{0,\tau(f)}\,,\\ & \ \ \ f(a^{\#}_{\sigma(f)})=a^\#_{\tau(f)} \forall a\in\kg\,,  f\hbox{ induces an isometry } T_{p(\sigma(f))}M\to T_{p(\tau(f))}M\big\}\,.	
\end{split}
\end{equation}

\begin{lm}\label{OpenLemma}
$S^{\bar\nabla}_{g,K}$ is $K$-invariant, and is a union of integral submanifolds of the distribution 	${\cal D}^{\bar \nabla}$. In particular $\Sg_{g,K}^{\bar\nabla}\coloneqq \delta^{-1}(S^{\bar\nabla}_{g,K})$ is a $K$-invariant open submanifold of $\Sg^{\bar\nabla}$.
\end{lm}
\begin{proof}
We have to prove that for any $f\in 	S^{\bar\nabla}_{g,K}$ one has ${\cal D}_f\subset T_fS^{\bar\nabla}_{g,K}$, i.e. that for any $\xi\in T_{\sigma(f)}P$ we have $\lambda_f(\xi)\in T_fS^{\bar\nabla}_{g,K}$. It suffices to prove that for any $a\in\kg$ and any $\zeta \in T_{p(\sigma(f))}M$ one has (denoting by $\tilde\zeta_{\sigma(f)}$ the $A$-horizontal lift of $\zeta$ at $\sigma(f)$):
$$\lambda_f(a^\#_{\sigma(f)})\in T_fS^{\bar\nabla}_{g,K}\,,\ \lambda_f(\tilde \zeta_{\sigma(f)})\in T_fS^{\bar\nabla}_{g,K}\,.
$$
The first formula is obtained using the curve $t\mapsto\sigma(f)\exp(ta)$, and the second is obtained using the curve $t\mapsto \tilde\eta_{\sigma(f)}(t) $, where $\eta:(-\varepsilon,\varepsilon)\to M$ is a $\nabla$-geodesic such that $\eta(0)=p(\sigma(f))$, $\dot\eta(0)=\zeta$. One uses the fact that the vector fields $a^\#$ are $\bar\nabla$-parallel along $A$-horizontal curves, and that the distributions $A$, $A_0$ are $\bar\nabla$-parallel. 
\end{proof}

\begin{lm}\label{SurjPairs}
The restrictions, 
$$\resto{(p\circ\sigma,p\circ\tau)}{S^{\bar\nabla}_{g,K}}:S^{\bar\nabla}_{g,K}\to M\times M\,,\ \resto{(p\circ\sg,p\circ\tg)}{\Sg^{\bar\nabla}_{g,K}}:\Sg^{\bar\nabla}_{g,K}\to M\times M$$
are  surjective.	
\end{lm}
\begin{proof}

Let  $x_0$, $x_1\in M$, and let $\eta:[0,1]\to M$ be a smooth path in $M$ such that $\eta(0)=x_0$, $\eta(1)=x_1$. Choose a point $y_0\in P_{x_0}$,  let $\tilde\eta$ be the $A$-horizontal lift of $\eta$ with the initial condition $\tilde\eta(0)=y_0$, and let $y_1\coloneqq \tilde\eta(1)$. Using $\bar\nabla$-parallel transport along $\tilde\eta$, we obtain an element $f\in S^{\bar\nabla}$ with $\sigma(f)=y_0$, $\tau(f)=y_1$.  Using Theorem \ref{ConnectionsCA}, we see  that  $f\in S^{\bar\nabla}_{g,K}$
\end{proof}

\begin{thry}\label{equivTh}
Let $(M,g)$ be a connected Riemannian manifold, $(g,P\textmap{p} M,A_0)$  be a triple consisting of a   principal $K$-bundle $P$ on $M$, and a connection $A_0$ on $P$. The following conditions are equivalent:
\begin{enumerate}
\item 	$(g,P\textmap{p} M,A_0)$  is locally homogeneous.
\item $(g,P\textmap{p} M,A_0)$  is infinitesimally homogeneous.
\item There exists a pair $(\nabla,A)$ consisting of a metric connection on $(M,g)$ and a connection $A\in {\cal A}(P)$ such that
\begin{equation}\label{ConditionsNablaA} \nabla  R^\nabla=0,\ \nabla T^\nabla=0,\ (\nabla\otimes \nabla^A)F^A=0,\  (\nabla\otimes \nabla^A)(A-A_0)=0\,.
\end{equation}
\end{enumerate}
 
\end{thry}

\begin{proof}  The implication (1)$\Rightarrow$(2) is obvious, and the implication (2)$\Rightarrow$(3) is stated by Theorem  \ref{nablaCA-FA-nabla-alpha}. For the implication (3)$\Rightarrow$(1), let $\bar\nabla=\nabla^{h,A}\oplus \nabla^{v,A}$ be the connection  on $T_P$ associated with a pair $(\nabla,A)$. Let $(x_0,x_1)\in M\times M$. By Lemma \ref{SurjPairs}  there exists $\varphi\in \Sg^{\bar\nabla}_{g,K}$ such that $y_0\coloneqq \sg(\varphi)\in P_{x_0}$,  $y_1\coloneqq \tg(\varphi)\in P_{x_1}$.  The orbit $\varphi K$ is a submanifold of $\Sg^{\bar\nabla}_{g,K}$ which is mapped diffeomorphically onto $y_0 K$ via $\sg$, and onto $y_1K$ via $\tg$. Regarding the maps $\sg$, $\tg: \Sg^{\bar\nabla}_{g,K}\to P$ as sheaves of sets over $P$, and using \cite[Théorème 3.3.1, p. 150]{God} we obtain an open neighbourhood $\Ug\subset \Sg^{\bar\nabla}_{g,K}$ of $\varphi K$ in  $\Sg^{\bar\nabla}_{g,K}$ which is  mapped injectively onto an open neighborhood ${\cal U}$  of $y_0 K$ via $\sg$, and is  mapped injectively onto an open neighborhood ${\cal V}$ of $y_1 K$ via $\tg$. Since $K$ is compact, we may suppose that $\Ug$ is $K$-invariant. Therefore   ${\cal U}$ and ${\cal V}$ will be also $K$-invariant, so  ${\cal U}=p^{-1}(U)$,   ${\cal V}=p^{-1}(V)$ for open neighborhoods $U$, $V$ of $x_0$, $x_1$ respectively. $\Ug$ defines a $K$-equivariant, $\bar\nabla$-affine isomorphism ${\cal U}\to {\cal V}$ which maps $A_{0U}$ onto  $A_{0V}$ and induces an isometry $ U \to V$. This shows that $(g,P\textmap{p} M,A_0)$ is locally homogeneous.
\end{proof}

For the symmetric case we have:
\begin{thry}
Let $(M,g)$ be a connected Riemannian manifold, $(g,P\textmap{p} M,A_0)$  be a triple consisting of a   principal $K$-bundle $P$ on $M$, and a connection $A_0$ on $P$.   The following conditions are equivalent:
\begin{enumerate}
\item 	$(g,P\textmap{p} M,A_0)$  is locally symmetric.
\item  One has
$$\nabla^{g}R^{g}=0,  (\nabla^{g}\otimes \nabla^{A_0}) F^{A_0}=0\,.
$$
\end{enumerate}
	
\end{thry}
\begin{proof}

Suppose that (1) holds, and, for a point $x\in M$, let $U\coloneqq U_x$, $s\coloneqq s_x$, $\Phi\coloneqq \Phi_x$ be as in Definition \ref{DefSymLH}. Taking into account that $\nabla^g$  is $s$-invariant, and $A_0$ is $\Phi$-invariant,   it follows that $\nabla^{g}R^{g}$, $(\nabla^{g}\otimes \nabla^{A_0}) F^{A_0}$ are invariant sections of $(\Lambda^1_U)^{\otimes 5}$,    $(\Lambda^1_U)^{\otimes 3}\otimes \ad(P_U)$ which are invariant with respect to the involution defined by $s$, respectively by the pair $(s,\Phi)$. Since the degree of these sections with respect to $M$ are odd, the argument of \cite[Section XI.1, Theorem 1.1]{KN2} applies.\\

Conversely, suppose (2) holds. Note that the pair $(\nabla^g,A_0)$ satisfies the conditions in Lemma \ref{AA0parallel}. Using the same construction as in the proof  of Theorem \ref{equivTh}, we obtain a linear connection $\bar\nabla$ on $P$ which satisfies the conditions $\bar \nabla R^{\bar\nabla}=\bar \nabla T^{\bar\nabla}=0$ and leaves the horizontal distribution of $A_0$ invariant (see Theorem \ref{ConnectionsCA}).

Let $x\in M$. For any point $y\in P_x$ put  
$$f_y\coloneqq \begin{pmatrix}
-\id_{(A_0)_y} & 0\\
0 &\id_{T_y(P_x)}	
\end{pmatrix}\in \GL(T_yP),$$
where 	$(A_0)_y\subset T_yP $ stands for the $A_0$-horizontal space at $y$. Using formulae  (\ref{TorsionNablaBar1}), (\ref{TorsionNablaBar2}), (\ref{CurvatureNablaBar}) it  is easy to check  that for any $y\in P_x$ one has $f_y\in S^{\bar\nabla}_{g,K}$.  Moreover, it is easy to see that $F_x\coloneqq \{f_y|\ y\in P_x\}$ is a $K$-orbit with respect to the $K$-action on  $S^{\bar\nabla}_{g,K}$. Using the bijection $\Sg^{\bar\nabla}_{g,K}\to S^{\bar\nabla}_{g,K}$ we obtain a $K$-orbit $\Fg_x\coloneqq \{\varphi_y|\ y\in P_x\}$ for the $K$-action on $\Sg^{\bar\nabla}_{g,K}$ which is mapped diffeomorphically onto $P_x$ via the maps $\sg$ and $\tg$. The same arguments as in the proof of Theorem \ref{equivTh} apply, and give a bundle isomorphism $\Phi: p^{-1}(U)\to  p^{-1}(V)$, where $U$, $V$ are open neighborhoods of $x$, and the induced map $s:U\to V$ is an isometry.   Moreover, one has $\Phi_{*y}=f_y$ for any $y\in P_x$, in particular $s_{*x}=-\id_{T_xM}$. Therefore $s$ induces an isometric involution on a sufficiently small normal neighborhood  $W\subset M$ of $x$. It suffices to note that the restriction of $\Phi$ to $p^{-1}(W)$ is an involution. This is obvious taking into account  the  next remark which gives an explicit geometric construction of $\Phi$: 

\begin{re} Let $\gamma:[0,1]\to W$ be a smooth path with $\gamma(0)=x$, $\gamma'\coloneqq s\circ\gamma$, let $y\in P$, and let $\tilde\gamma$, $\tilde\gamma'$ be the $A_0$-horizontal lifts with initial conditions $\tilde\gamma(0)=\tilde\gamma'(0)=y$. Then $\Phi(\tilde\gamma(1))= \tilde\gamma'(1) $, and $\Phi(\tilde\gamma'(1))= \tilde\gamma(1)$.
\end{re} 	
 
 The remark is proved using that $\Phi$ lifts $s$ and leaves the connection $A_0$ invariant.

\end{proof}

For the case when $(M,g)$ is complete, we have

\begin{thry} \label{CompleteBase} Let $(g,P\textmap{p} M,A_0)$ be a  locally homogeneous triple with $M$ connected.
If $(M,g)$ is complete, then the map $\Sg^{\bar\nabla}_{g,K}/K\to M$ induced by $\sg$  is a covering map. If $(M,g)$ is complete  and $M$ is simply connected, then  any  germ $\varphi\in \Sg^{\bar\nabla}_{g,K}$ can be extended to a unique bundle isomorphism  $\Phi:P\to P$ which covers an isometry $M\to M$,  and has the property $\Phi_*(A_0)=A_0$. In particular, for any $(x,x')\in M\times M$   there exists such a bundle isomorphism with $\Phi(P_x)=P_{x'}$.	
\end{thry}
\begin{proof}
We use the same method as in the proof of Proposition  \ref{globalAff}, Corollary  \ref{globalAffg}. The fact that $\Sg^{\bar\nabla}_{g,K}/K\to M$ is a covering map is obtained using parallel transport with respect to $\bar\nabla$ along $A$-horizontal lifts of $\nabla$-geodesics in $M$.
\end{proof}

\begin{co} \label{CompleteBasecor}
Let $(g,P\textmap{p} M,A_0)$  be a locally symmetric	 triple with $M$ simply connected, and $(M,g)$ complete. Then $(g,P\textmap{p} M,A_0)$  is a symmetric triple.
\end{co}
\begin{proof}
By Theorem 	\ref{equivTh} any locally symmetric triple is locally homogeneous, so Theorem \ref{CompleteBase} applies, so the locally defined involutive bundle isomorphisms $\Phi_x$ given by the  locally symmetric	condition (see Definition \ref{DefSymLH}) extend to the whole total space $P$.
\end{proof}

Using theorem \ref{CompleteBase} we can also prove the following classification theorem for LH triples:

\begin{thry} \label{main-DIffcase} Let $M$ be a compact manifold, and $K$ be a compact Lie group. Let $\pi:\tilde M\to M$ be the universal cover of $M$, $\Gamma$ be the corresponding covering transformation group. Then, for any   locally homogeneous triple $(g,P\textmap{p} M,A)$  with structure group $K$ on $M$ there exists
\begin{enumerate}
\item 	A connection $B$ on the pull-back bundle $Q\coloneqq \pi^*(P)$.
\item   A  closed subgroup $G\subset \Iso(\tilde M,\pi^*(g))$ acting transitively  on $\tilde M$ which contains $\Gamma$ and leaves  invariant the gauge class $[B]\in {\cal B}(Q)$.
\item A lift $\jg:\Gamma\to {\cal G}_G^B(Q)$ of  the inclusion monomorphism $\iota_\Gamma:\Gamma\to G$, where ${\cal G}_G^B(Q)$ stands for the group of automorphisms of $(Q,B)$ which lift transformations in $G$.
\item An isomorphism between the $\Gamma$-quotient of $(\pi^*g,Q\to \tilde M,B)$ and the initial triple $(g,P\textmap{p} M,A)$ .

\end{enumerate}
\end{thry}

\begin{proof}
Let $G\subset \Iso(\tilde M,\tilde g)$ be the subgroup defined by
$$G\coloneqq \big\{\psi\in \Iso(\tilde M,\tilde g)|\ \exists \,\Psi:Q\to Q \hbox{ $\psi$-covering bundle isom., }\Psi^*(B)=B\big\}.
$$
Using the fact that $K$ is compact, it follows by Lemma \ref{PropClosed} below, that $G$ is a closed subgroup of the Lie group $\Iso(\tilde M,\tilde g)$. Note that Lemma \ref{PropClosed} applies because the action of the Lie group $\Iso(\tilde M,\tilde g)$ on $\tilde M$ is smooth.
 \vspace{1.5mm}
 
Applying Theorem  \ref{CompleteBase} to  $(\pi^*(g),Q,B)$, it follows that $G$ acts transitively on $\tilde M$, and leaves invariant the gauge class $[B]$. Moreover, the definition of $G$ shows that it contains $\Gamma$. The lift $\jg$ is obtained as follows: for $\varphi\in \Gamma$ we define $\jg(\varphi):Q\to Q$ by 
$$\jg(\varphi)(\tilde x,y)\coloneqq (\varphi(\tilde x),y)\ \forall (\tilde x,y)\in Q\coloneqq \tilde M\times_\pi P\,.
$$
Note that the map $\Gamma\ni\varphi\mapsto \jg(\varphi)\in {\cal G}_G^B(Q)$ is a group morphism (as required). 
\end{proof}

\begin{lm}\cite{Ba1}\label{PropClosed}
Let $M$ be a differentiable manifold, $K$ be a compact Lie group, $p:P\to M$ be a  principal $K$-bundle  over $M$, and  $A\in{\cal A}(P)$ be a connection on $P$. Let $\alpha:L\times M\to M$ be a a smooth action of a Lie group	$L$ on   $M$. For $l\in L$ denote by $\varphi_l:M\to M$ the corresponding diffeomorphism. The subspace
$$L_A\coloneqq \{l\in L|\ \exists \Phi\in\Hom_{\varphi_l}(P,P)\hbox{ such that }\Phi^*(A)=A\} 
$$
is a closed Lie subgroup of $L$.
\end{lm}

 Theorem \ref{main-DIffcase} shows that:
\begin{co}\label{main-coro} Let $M$ be a connected compact   manifold, and $K$ be a compact Lie group.   Let $\pi:\tilde M\to M$ be the universal cover of $M$, $\Gamma$ be the corresponding covering transformation group.
Then any    locally homogeneous triple $(g,P\textmap{p} M,A)$  with structure group $K$ on $M$ can be   identified with a $\Gamma$-quotient of a  homogeneous triple $(\pi^*g,Q\coloneqq \pi^*(P)\to \tilde M, B)$ on the universal cover $\tilde M$.	
\end{co}

\subsection{Locally homogeneous $\Spin^c$-pairs and triples}

The group  $\Spin^c(n)\coloneqq\Spin(n)\times_{\Z_2}\S^1$ fits in a short exact sequence
$$
1\to \{\pm 1\} \to \Spin^c(n)\textmap{(\rho,\delta)}\SO(n)\times\S^1\to \{1\}\,,
$$  
where $\rho$ is induced by the canonical morphism $r:\Spin(n)\to \SO(n)$, and $\delta$ is induced by $\Spin(n)\times \S^1\ni  (u,z)\mapsto z^2$.  
Let $(M,g)$ be a connected, oriented Riemannian $n$-manifold. A $\Spin^c$-structure on $M$ is a bundle morphism $\Lambda:Q\to \SO(M)$ of type $\rho$, where $Q$ is a $\Spin^c(n)$-bundle.  The associated $\S^1$-bundle $\delta(Q)\coloneqq Q\times_\delta\S^1$ is called the determinant bundle  of $Q$ (or of $\Lambda$). The spinor bundle of $\Lambda$ is the associated bundle $S\coloneqq Q\times_\kappa\Delta_n$, where $\kappa:\Spin^c(n)\to \GL(\Delta_n)$ is the natural extension of the $\Spin$ representation.

We recall (see for instance \cite[Remark, p. 48]{Fr}, \cite[Proposition 5.2.7 (ii)]{Te}) that $(M,g)$ admits a $\Spin^c$-structure if and only if the Stiefel-Whitney class $w_2(M)$ admits an integral lift, and that, if non-empty, the set of isomorphism classes of  $\Spin^c$-structures on $M$ is an $H^2(M,\Z)$-torsor (see \cite[Proposition p. 53]{Fr}, \cite[Proposition 5.2.7 (iii)]{Te}). 

Whereas the principal bundle of a $\Spin$-structure comes with a unique lift of the Levi-Civita connection $C_0\in{\cal A}(\SO(M))$, for a $\Spin^c$-structure  we have:

\begin{re}\label{spin-spinc} Let  $\Lambda:Q\to \SO(M)$ be a $\Spin^c$-structure on $(M,g)$. 
The map ${\cal A}(Q)\ni A\mapsto \delta(A)\in {\cal A}(\delta(Q))$ induces an isomorphism between the space of lifts of $C_0$ in ${\cal A}(Q)$ and the space  ${\cal A}(\delta(Q))$ of connections on the determinant $\S^1$-bundle $\delta(Q)$.
 \end{re} 
 \begin{re}\label{spincDirac}
 These Levi-Civita lifts are important because they intervene in the construction of $\Spin^c$-Dirac operators. The Dirac operator $\Dr_A$ associated with a lift $A$ of $C_0$ is defined as the composition 
 $$
\Gamma(M,S)\textmap{\nabla_{A}}\Gamma(M,\Lambda^1_M\otimes S)\textmap{c} \Gamma(M,S)\,,
$$
where   $c$ is induced by the Clifford multiplication $\Lambda^1_M\otimes S\to S$ (see for instance \cite[Section 3.2]{Fr}, \cite[section 5.3]{Te}). $\Dr_A$ is a first order, elliptic, self-adjoint operator. 
\end{re}

Note that the assignment $A\mapsto\slashed{D}_A$ plays a fundamental role in Seiberg-Witten theory (see for instance \cite[section 6.1]{Te}). Remark \ref{spincDirac} justifies the definition
\begin{dt}
Let $(M,g)$ be a compact, connected, oriented Riemannian $n$-manifold. A $\Spin^c$-pair on $(M,g)$ is a pair $(\Lambda,A)$, where $\Lambda: Q\to \SO(M)$ is a $\Spin^c$-structure on $(M,g)$, and $A\in {\cal A}(Q)$ is a lift of the Levi-Civita connection.
\end{dt}
In other words a $\Spin^c$-pair on $(M,g)$ is a system of data which intervene in the construction of a $\Spin^c$-Dirac operator. 
For such pairs we have a natural local homogeneity condition
\begin{dt}\label{defLHSpinc}
Suppose that $(M,g)$ is locally homogeneous. A $\Spin^c$-pair $(\Lambda,A)$ on $(M,g)$ is called LH   if for any  pair $(x_0,x_1)\in M\times M$ there exists a local isometry $x_0\in U_0\textmap{\varphi} U_1\ni x_1$, and an $\tilde \varphi_*$-covering bundle isomorphism $Q_{U_0}\textmap{\Phi}  Q_{U_1}$ which preserves $A$.
\end{dt}

\begin{pr}
Let $(\Lambda,A)$ be a $\Spin^c$-pair on  a locally homogeneous Riemannian manifold $(M,g)$, let $a\in {\cal A}(\delta(Q))$ be the associated $\S^1$-connection, and $F_a\in iA^2(M)$ be its curvature form. The following conditions are equivalent:
\begin{enumerate} 
\item  $(\Lambda,A)$ is 	locally homogeneous.
\item The triple $(g, \delta(Q)\to M,a)$ is locally homogeneous.
\item for any  $(x_0,x_1)\in M\times M$ there exists a local isometry $x_0\in U_0\textmap{\varphi} U_1\ni x_1$ such that $\varphi^*(F_a|_{U_1})=F_a|_{U_0}$.
\item $(M,g)$ admits an Ambrose-Singer connection $\nabla$ such that $\nabla F_a=0$.
\end{enumerate}
	
\end{pr}
\begin{proof} The implications (1)$\Rightarrow$(2)$\Rightarrow$(3) are obvious. The equivalence (3) $\Leftrightarrow$(4) is a special case of Theorem \ref{KiriTh}. The implication (4) $\Rightarrow$ (2) follows from Theorem \ref{equivTh} choosing $A_0=A=a$.  Indeed,  the adjoint bundle $\ad(\delta(Q))$ can be identified with the trivial line bundle $i\underline{\R}$, and the  connection  induced on $\ad(\delta(Q))$ by any connection on $\delta(Q)$ is the standard trivial connection. Therefore the condition $(\nabla\otimes \nabla^a)F_a=0$ intervening in  Theorem \ref{equivTh}  becomes $\nabla F_a=0$.

It remains to prove the implication (2)$\Rightarrow$(1). Let $(x_0,x_1)\in M\times M$. We know that there exists an isometry $x_0\in U_0\textmap{\varphi} U_1\ni x_1$ and a $\varphi$-covering bundle isomorphism $\phi: \delta(Q)_{U_0}\to  \delta(Q)_{U_1}$ which preserves $a$. We obtain a bundle isomorphism 
$$(\tilde\varphi_*,\phi):(\SO(M)\times_M \delta(Q))_{U_0}\to (\SO(M)\times_M \delta(Q))_{U_1}$$
which preserves the connection $(C_0,a)$.  Assuming $U_i$ simply connected, there exists a bundle isomorphism $\Phi:Q_{U_0}\to Q_{U_1}$ which lifts $(\tilde\varphi_*,\phi)$, in particular $\Phi$ lifts $\tilde\varphi_*$ via $\Lambda$. The connections $A_{U_0}$, $\Phi^*(A_{U_1})$ coincide, because they both lift the connection $(C_0,a)_{U_0}=(\tilde\varphi_*,\phi)^*((C_0,a)_{U_1})$ on $(\SO(M)\times_M \delta(Q))_{U_0}$.
\end{proof}
Applying Theorem \ref{CompleteBase} to the triple $(g, \delta(Q)\to M,a)$ and using the same lifting argument as above we obtain:
\begin{co}
Let $(\Lambda,A)$ be an LH $\Spin^c$-pair on  a  complete, simply connected LH Riemannian manifold $(M,g)$. 	There exists a Lie group $G$ and an action $\beta : G \times Q \to Q$ by bundle isomorphisms with the following properties:
\begin{enumerate}
\item 	$\beta$ lifts a transitive action by isometries $\alpha : G \times M \to  M$.
\item $\beta$ leaves $A$   invariant.
\end{enumerate}
\end{co}
One has also a natural local homogeneity condition for triples $(\Lambda,A,s)$, where $(\Lambda,A)$ is a $\Spin^c$-pair on $(M,g)$, and $s\in \Gamma(M,S)$. Such a triple will be called a $\Spin^c$-triple.
\begin{dt} 
 Let $(M,g)$ be an LH  Riemannian manifold. A    $\Spin^c$-triple $(\Lambda,A,s)$ is  called LH  if for any pair $(x_0,x_1)\in M\times M$ there exists a local isometry $x_0\in U_0\textmap{\varphi} U_1\ni x_1$, and a  $\tilde \varphi_*$-covering bundle isomorphism $Q_{U_0}\textmap{\Phi}  Q_{U_1}$ which preserves $A$ and $s$. 
 \end{dt} 
 
 Note that if  $(\Lambda,A,s)$ is LH, so is $(\Lambda,A,\Dr_As)$.
\begin{pr}\label{PropExB}
Let $(M,g)$ be 	LH  Riemannian manifold, and $(\Lambda,A,s)$  be an LH  $\Spin^c$-triple on  $(M,g)$. There exists   $B\in {\cal A}(Q)$ such that
\begin{enumerate}
\item The induced connection $C\coloneqq \rho(B)$ is an Ambrose-Singer connection on the frame bundle $\SO(M)$.
\item  Putting $a\coloneqq \delta(A)$, $b\coloneqq \delta(B)$, one has $\nabla^C F_a=0$, $\nabla^C(b-a)=0$.
\item $\nabla^B s=0$.
\end{enumerate}

\end{pr}
\begin{proof}
The local homogeneity assumptions imply that the  section defined by the triple $(R^g,F_a,s)$  is infinitesimally homogeneous with respect to the connection $A$. The claim follows from Corollary \ref{ConnectionBCo} noting that, via the direct sum decomposition $\ad(Q)=\so(T_M)\oplus i\underline{\R}$, one has $B-A=(C-C_0,b-a)$, $F_A=(F_{C_0},F_a)$.
\end{proof}

Using the  proofs of Lemma \ref{SurjPairs},  Theorems \ref{equivTh}, \ref{CompleteBase} we obtain the  following analogue of Corollary \ref{spin-coro} for LH  $\Spin^c$-triples:

\begin{co} Let $(M, g)$ be a simply connected, complete LH Riemannian manifold, and $(\Lambda,A,s)$  be an LH  $\Spin^c$-triple on  $(M,g)$. 
There exists a Lie group $G$ and an action $\beta : G \times Q \to Q$ by bundle isomorphisms with the following properties:
\begin{enumerate}
\item 	$\beta$ lifts a transitive action by isometries $\alpha : G \times M \to  M$.
\item $\beta$ leaves $A$ and $s$ invariant.
\end{enumerate}

\end{co}

\begin{proof} Let $B\in {\cal A}(Q)$ be a connection given by  Proposition \ref{PropExB}, and put $C\coloneqq\rho(B)$, $b\coloneqq\delta(B)$, $a\coloneqq \delta(A)$. Applying Lemma \ref{AA0parallel} to the triple $(g,\delta(Q)\to M,a)$ and the pair $(\nabla^C,b)$, we obtain a connection $\bar\nabla$ on the total space $\delta(Q)$ which preserves the distributions $a$ and $b$. 

 It suffices to prove that for any $(x_0,x_1)\in M\times M$ there exists an isometry $\varphi:M\to M$ with $\varphi(x_0)=x_1$, and a $\tilde\varphi_*$-covering bundle isomorphism $\Phi:Q\to Q$ which preserves $A$ and $s$.  Since the triple $(g, \delta(Q)\to M,a)$ is locally homogeneous, Theorem \ref{CompleteBase} gives an isometry $\varphi:M\to M$ and a $\varphi$-covering bundle isomorphism $\phi:\delta(Q)\to \delta(Q)$ which preserves  $a$. 
 
 Using the constructive proof of the transitivity property stated in  Lemma \ref{SurjPairs}, we see that such a lift $\phi$ can be obtained explicitly in the following way:
 
Let $\eta:[0,1]\to M$ be  a smooth path with $\eta(i)=x_i$,  $\tilde \eta:[0,1]\to \delta(Q)$ be a $b$-horizontal lift of $\eta$, and let $f:T_{\tilde\eta(0)}(\delta(Q))\to T_{\tilde\eta(1)}(\delta(Q))$ be the linear isomorphism given by $\bar\nabla$-parallel transport along $\tilde \eta$. 

The proofs of  Theorems \ref{equivTh}, \ref{CompleteBase} show that $f$ is the tangent map at $\tilde \eta(0)$ of a global bundle isomorphism $\phi: \delta(Q)\to \delta(Q)$ which is affine with respect to $\bar \nabla$ and lifts an isometry $\varphi:M\to M$ with $\varphi(x_0)=x_1$. Taking into account the definition of $f$ and the explicit construction of $\bar\nabla$, it follows that the tangent map $\varphi_{*x_0}$ is given by $\nabla^C$-parallel transport along $\eta$, so $\tilde\varphi_{*,x_0}$ is given by $C$-parallel transport along $\eta$ in $\SO(M)$.

Since $\bar \nabla$  preserves $a$ and $b$, $\phi_{*,\tilde\eta(0)}$ is given by $\bar\nabla$-parallel transport, and $\Phi$ is $\bar\nabla$-affine, it follows that $\phi$ preserves the connections $a$ and $b$. The bundle isomorphism
$$
(\tilde\varphi_*,\phi):\SO(M)\times_M \delta(Q)\to \SO(M)\times_M \delta(Q)
$$
preserves the connection $(C_0,a)$, so any lift   $\Phi:Q\to Q$ of it preserves $A$.

It suffices to prove that there exists such a lift which also preserves $s$. Let $\eta_1:[0,1]\to \SO(M) $ be a $C$-horizontal  lift of $\eta$, and $\bar \eta$ be a lift of $(\eta_1,\tilde\eta)$ to $Q$ via the double cover $Q\to \SO(M)\times_M \delta(Q)$.  We know that 
$$(\tilde\varphi_*,\phi)(\eta_1(0),\tilde\eta(0))=(\eta_1(1),\tilde\eta(1))\,.$$

 Since $M$ is simply connected, the double cover $Q\to \SO(M)\times_M \delta(Q)$ is trivial, so there exists a lift $\Phi:Q\to Q$   of $(\tilde\varphi_*,\phi)$ which maps $\bar \eta(0)$ onto $\bar \eta(1)$. We claim that $\Phi(s)=s$, or, equivalently, that $\tilde s\circ \Phi=\tilde s$, where $\tilde s: Q\to \Delta_n$ is the equivariant map associated with $s$.

Since $(\eta_1,\tilde\eta)$ is $(C,b)$-horizontal, it follows that its lift $\bar \eta$ to $Q$ is $B$-horizontal. Since $\nabla^Bs=0$ it follows that $\tilde s$ is constant along $\bar\eta$, so 
$$\tilde s(\bar \eta(0))=\tilde s(\bar \eta(1))=(\tilde s\circ \Phi)(\bar \eta(0))\,.$$
Therefore $\tilde s$,  $\tilde s\circ \Phi$ are equivariant maps which are constant along $B$-parallel curves and coincide at  $s(\bar \eta(0))$, so they coincide.
 
 \end{proof}

\end{document}